\documentclass{amsart}

\usepackage{amssymb}
\usepackage{latexsym}
\usepackage{amsmath}
\usepackage{enumerate}
\usepackage{amsthm}
\usepackage{wasysym}
\usepackage{stmaryrd}
\usepackage{mathrsfs}
\usepackage[usenames]{xcolor}
\usepackage{pifont}
\usepackage{tikz}
\usepackage{cases}

\newcommand{\qee} {\hspace*{2mm}\hfill \ding{109}}

\renewcommand{\iff}{\leftrightarrow}
\renewcommand{\phi}{\varphi}

\definecolor{uuxgreen}{cmyk}{1,0,0.75,0}
\definecolor{uuxred}{cmyk}{0.2,1,0.9,0.1}
\definecolor{uuyblue}  {cmyk}{0.9,0.55,0,0}

\newcommand{\qedright}{\belowdisplayskip=-12pt}

\newtheorem{theorem}{Theorem}[section]
\newtheorem{define}[theorem]{Definition}

\newtheorem{exa}[theorem]{Example}
\newenvironment{example}{\begin{exa} \rm}{\qee\end{exa}}
\newtheorem{exerc}[theorem]{Exercise}

\newtheorem{conj}[theorem]{Conjecture}
\newenvironment{conjecture}{\begin{conj} \rm}{\qee\end{conj}}
\newtheorem{ques}[theorem]{Open Question}
\newenvironment{question}{\begin{ques} \rm}{\qee\end{ques}}
\newtheorem{lem}[theorem]{Lemma}
\newenvironment{lemma}{\begin{lem} \it}{\end{lem}}
\newtheorem{cor}[theorem]{Corollary}
\newenvironment{corollary}{\begin{cor} \it}{\end{cor}}
\newtheorem{rem}[theorem]{Remark}
\newenvironment{remark}{\begin{rem} \rm}{\qee\end{rem}}

\DeclareMathOperator{\possible}{\text{\tikz[scale=.6ex/1cm,baseline=-.6ex,rotate=45,line width=.1ex]{
                            \draw (-1,-1) rectangle (1,1);}}}
\DeclareMathOperator{\necessary}{\text{\tikz[scale=.6ex/1cm,baseline=-.6ex,line width=.1ex]{
                            \draw (-1,-1) rectangle (1,1);}}}

 \DeclareMathOperator{\gnecessary}{\text{\tikz[scale=.6ex/1cm,baseline=-.6ex,line width=.1ex]{
                            \draw[gray, fill = gray, fill opacity = .90] (-1,-1) rectangle (1,1);}}}

 \newcommand{\tupel}[1]{{\langle #1 \rangle}}
\newcommand{\verz}[1]{\{ #1 \}}
\newcommand{\To}{\Rightarrow}

\newcommand{\gnum}[1]{{\ulcorner #1 \urcorner}}
\newcommand{\gn}[1]{{\underline{\ulcorner #1 \urcorner}}}

\newcommand{\apr}{{\vartriangle}}
\newcommand{\aco}{{\triangledown}}
\newcommand{\opr}{\necessary}

\newcommand{\oco}{\possible}
\newcommand{\jump}{\mathrel{\mbox{\textcolor{gray}{$\blacktriangleright$}}}}
\newcommand{\jumpb}{\mathrel{\mbox{\textcolor{gray}{$\blacktriangleleft$}}}}

\newcommand{\graysq}{\gnecessary}
\newcommand{\stackarrow}[1]{\stackrel{#1}{\longrightarrow}}
 \newcommand{\nrhd}{\mathrel{\not\! \rhd}}

\newcommand{\Su}{\mathrel{\sf S}}
\newcommand{\slinei}{\raise-0.3ex\hbox{$\hbox{--}\kern-0.84ex\raise0.45ex\hbox{$\hbox{\scalebox{0.3}
{\bf \textbackslash}}\kern-0.37ex\hbox{\scalebox{0.3}{\bf \textbackslash}}$}$}}
\newcommand{\sline}{\raise-0.3ex\hbox{$\hbox{--}\kern-0.84ex\raise0.45ex\hbox{$\hbox{\scalebox{0.3}{\bf /}}\kern-0.37ex\hbox{\scalebox{0.3}{\bf /}}$}$}}
\newcommand{\lhdnneq}{\mathrel{\lhd_{\hspace*{-0.27cm}{}_{\kern0.2ex \slinei}}\hspace*{0.09cm}}}
\newcommand{\jumpneq}{\mathrel{\jump_{\hspace*{-0.235cm}{}_{\kern0.08ex \sline}}\hspace*{0.09cm}}}
\newcommand{\jumpbneq}{\mathrel{\jumpb_{\hspace*{-0.235cm}{}_{\kern0.08ex \slinei}}\hspace*{0.09cm}}}
\newcommand{\mutfuj}{\mathrel{\text{\small ${\jump}\hspace*{-0.145cm}{\jumpb}$}}}
\newcommand{\mutint}{\bowtie}
\newcommand{\sat}{{\sf sat}}

\title[Enayat Theories]{Enayat Theories}

\author{Albert Visser}
 \address{Philosophy, Faculty of Humanities,
                Utrecht University,
               Janskerkhof 13,
                3512BL~~Utrecht, The Netherlands}
\email{a.visser@uu.nl}
\date{\today}

\begin{document}



\thanks{I thank Ali Enayat for his inspiring question. I am grateful to Ali Enayat, Volker Halbach, Mateusz {\L}e{\l}yk,
Bartosz Wcis{\l}o and Fedor Pakhomov for inspiring discussions.}

\begin{abstract}
In this paper we study solution attempts for a problem posed by Ali Enayat:
can there be a finitely axiomatized consistent sequential theory that interprets itself
plus the (sentential or non-uniform) Tarski biconditionals? We provide a basic framework for the study of this question
and discuss some solution attempts. We connect the question with some interesting
conjectures. We briefly touch upon what happens if we consider uniform biconditionals.
\end{abstract}

\maketitle

\section{Introduction: Ali Enayat asks a question}
On January 24, 2014, Ali Enayat sent me an e-mail with subject `a simple (?!) question'.
The content of the mail was as follows.
\begin{quote}{\footnotesize
Suppose {\sf B} is a finitely axiomatizable base theory , and ${\sf B^T}$ is {\sf B} plus the T-scheme, i.e., biconditionals of the form:
 $S$ iff ${\sf T}(\# S)$, where $\# S$ is the code for $S$.
Question: Is ${\sf B^T}$ interpretable in {\sf B}?}
\end{quote}

\noindent
I thought that I would be able to solve the problem, for the most salient case where the base theory is sequential,
 within a day's time, but I was sadly mistaken. 
The problem is still unsolved for the sequential case and even for the more inclusive Vaught case. 
I give a positive example of a non-Vaught theory that interprets itself plus Tarski biconditionals in Section~\ref{blijesmurf}. 
Admittedly, this example involves a not quite standard G\"odel numbering.

 Let us say that a theory $U$ that interprets
$U$ plus the Tarski biconditionals for the language of $U$ is an \emph{Enayat theory}.\footnote{As explained in 
Section~\ref{badein}, it is somewhat more subtle to get the question right. The notion of Enayat theory \emph{tout court} only makes
sense is the case one considers Vaught theories.}
Ali's question suggests the following conjecture.
\begin{conjecture}
There are no consistent, finitely axiomatized, sequential Enayat theories.
\end{conjecture}
\noindent The argument in favor of the conjecture is simply the rhetorical question: \emph{what could such an interpretation
possibly look like?}

Why publish my failed attempts? I think there are some good reasons to do that.

\begin{itemize}
\item It is good to make people aware of the problem. It is a basic and intricate problem from the 
logico-technical standpoint. I think the related
problems concerning the meta-mathematics of first-order theories
 formulated in this paper illustrate that our problem leads to interesting further questions. 

Also, the problem fits, at the lower end, in the broader logico-philosophical program of research into truth theories.
It is, in a sense, about the informativeness of the minimal typed truth theory.

Along a different line, I think it is time logicians would look a bit more at sentential schemes (in some broad sense).
These often behave differently from their uniform brethren. The primary example of a sentential scheme is
parameter-free induction. See, e.g., \cite{kaye:para88}, \cite{bekl:para96} and \cite{cord:para11}. For a sllightly different, but
related, study, see \cite{viss:pean14}.
\item The paper provides the basic framework for the study of Ali's question. It is good to have these things out of the way.
\item If a reader would want to try her hand on the problem, the attempts contained in the paper would at least
spare her the time to rediscover those.
\item
In Section~\ref{propre}, we develop Saccheri style what an `Enayat world' would look like. This provides
some further basics that can play a role in a solution.
\item
In the study of this problem, errors are everywhere dense. Especially, one has to keep the dependencies of the
complexities of the various items involved in an argument straight. (I often had the illusion of having solved
the problem for days, but then a subtle circularity of dependencies turned up.) The paper provides, I hope, an example of
good practice in keeping track of dependencies.
\item
I feel some of the arguments in the paper are definitely entertaining. A good example is the proof of Theorem~\ref{grotesmurf}.
\end{itemize}

\noindent
In Section~\ref{neighbours}, we have a brief look at what happens when we consider uniform biconditionals.
In a subsequent paper, I hope to give a fuller picture.

\section{Basics}
In this section, we provide the basic framework for the study of Enayat theories.

\subsection{Theories, translations and interpretations}
Theories are, in our paper, theories of first order predicate logic of finite signature that
are given by a sufficiently simple set of axioms, say $\Delta_1^{\sf b}$.
In the few cases where we diverge from this format it will be explcitly mentioned.
  The axiom set is part of the data of a theory.

We refer the reader for a discussion of translations and interpretations to one of our papers
 \cite{viss:what13} or \cite{viss:onq17} or \cite{viss:smal18} or \cite{viss:inte18}. Here we just fix some notations.
 \begin{itemize}
 \item
 We write $U \rhd V$ for $U$ interprets $V$. 
 \item
 We write $\Gamma \rhd_U \Delta$ for $\Gamma$ interprets $\Delta$ over $U$, i.o.w., $(U+\Gamma) \rhd (U+\Delta)$.
 Here $\Gamma$ and $\Delta$ will be, in the typical case, sets of sentences,
 each with a signature that is an extension of the signature of $U$.
 \item
 We write $\Gamma \jump_U \Delta$ for $\Gamma$ Fujimoto interprets $\Delta$ over $U$.
 This means that we have an identity-preserving, unrelativised interpretation of $U+\Delta$ in $U+\Gamma$ that
 preserves the vocubulary of $U$. Here $\Gamma$ and $\Delta$ have signatures that extend the signature of $U$. 
 \item
 $\rhd_{\sf loc}$ stands for local interpretability and $\rhd_{\sf dir}$ stands for direct interpretability. Direct interpretability
 is unrelativized and identity preserving interpretability. We allow more-dimensional direct interpretability.
 \item
 We use $\mutint$ for mutual interpretability and $\mutfuj$ for mutual Fujimoto interpretability.\footnote{In my earlier papers, I use
 $\equiv$ for mutual interpretability. However, Lev Beklemishev uses this symbol for other notions of sameness of theories.
 I think the present notation will eliminate all ambiguity.}
 \end{itemize}
 
 \noindent
 Some knowledge
 of the book~\cite{haje:meta91} is definitely useful.

\subsection{Vaughtness and sequentiality}
In this subsection we define Vaught theories and sequential theories.
We refer the reader to \cite{viss:pair08} and \cite{viss:what13} for more information.

\subsubsection{Vaught Set Theory}
We define \emph{Vaught set theory}, {\sf VS} as follows.
\begin{enumerate}[{\sf VS}1.]
\item
$\exists x\, \forall y\, y \not\in x$,
\item
$\forall u_0\dots \forall u_{n-1}\, \exists x\, \forall y\, (y\in x \iff \bigvee_{i<n} y=u_i)$.
\end{enumerate}

\noindent
We note that, under the right conventions, {\sf VS}1 is the special case for $n=0$ of {\sf VS}2. We also note that we do not have extensionality.

We can define Kuratowski pairing in the usual way. Of course, our pairs will not be extensional and the same pair could be implemented by
many entities.
We define a function as follows:
\begin{itemize}
\item
$f$ is a function iff $\forall u\in f\, \forall v\in f\, ({\sf pair}(u) \wedge {\sf pair}(v) \wedge ((u)_0=(v)_0 \to u=v))$.
\end{itemize} 
We do not just demand the uniqueness of the output but also the uniqueness of the pair that
implements a transition. We define $f: x\sim y$ iff $f$ is a bijection between $x$ and $y$, and
$x\sim y$ iff $\exists f\, f:x\sim y$.

We define ${\sf VS}^+$ as {\sf VS} plus the axioms saying that $\sim$ is an equivalence relation and that
if $f:x\sim y$, then $x\sim f$.

We have the following theorem.

\begin{theorem}\label{brutalesmurf}
${\sf VS} \rhd_{\sf dir} {\sf VS}^+$. Moreover, the relevant interpretation is one-di\-men\-sion\-al.
\end{theorem}

\noindent We give the proof in Appendix~\ref{nakaartsmurf}.
Let {\sf R} be the very weak arithmetical theory given formulated by Tarski, Mostowski and Robinson in their classic
 \cite{tars:unde53}.
We  have:

\begin{theorem}\label{motorsmurf}
${\sf VS} \rhd {\sf R}$. Moreover, the relevant interpretation is one-dimensional.
\end{theorem}

\noindent We give the proof in Appendix~\ref{nakaartsmurf}.

A theory $U$ is a \emph{Vaught theory} if $U \rhd_{\sf dir} {\sf VS}$. A theory $U$ is a  
\emph{Vaught${}^+$ theory} if   $U \rhd_{\sf dir} {\sf VS}^+$. 
In these definitions, we allow the direct interpretation to be more-dimensional.

Theorem~\ref{brutalesmurf}
tells us that a Vaught theory is \emph{ipso facto} a Vaught${}^+$ theory.

\subsubsection{Adjunctive Set Theory}
We define \emph{adjunctive set theory}, {\sf AS}, as follows:
\begin{enumerate}[{\sf AS}1.]
\item
$\exists x\, \forall y\, y \not\in x$,
\item
$\forall u\,\forall v\, \exists x\, \forall y\, (y\in x \iff (y\in u \vee y=v))$.
\end{enumerate}

\noindent
A theory $U$ is \emph{sequential} if $U \rhd_{\sf dir} {\sf AS}$.
Here we allow the direct interpretation to be more dimensional.\footnote{Thus, our notion of sequentiality is
an extension of the usual one. In  \cite{viss:what13}, I called this notion \emph{polysequentiality}. However, in the
light of the facts that (i) we are just looking at a minor extension of the notion and (ii) the modified notion is clearly the right one, 
it is high time to redefine
the traditional notion.}
Sequentiality is well studied: we refer the reader to \cite{viss:what13} for an extensive discussion of the notion.

We present some basic facts concerning sequential theories.
We define (for any theory $U$):
\begin{itemize}
\item
 $\mho(U):= {\sf S}^1_2+ \verz{\oco_{U,n}\top \mid n\in \omega}$.
 \item
 A theory $U$ is \emph{reflexive} if $U \rhd \mho(U)$. 
 \end{itemize}
Here the $\oco_{U,n}\top$ are restricted consistency statements for $U$, where of course $U$ is
given by a fixed representation of the axioms and the restriction in both a restriction of the size of the codes of the
axioms and of the complexity of the formulas allowed in the proofs. The measure of complexity here is \emph{depth of quantifier
alternations}. 
 See, e.g., \cite{viss:seco11} or \cite{viss:inte18}.
 
 Here are some basic facts. 
\begin{enumerate}[1.]
\item
$\mho(U) \rhd U$ (this holds for any $U$).
\item
If $U$ is sequential, then $U \rhd_{\sf loc} \mho(U)$. I.o.w., sequential theories are \emph{locally} reflexive.
\item
If $U$ is reflexive, then $U \rhd_{\sf loc} V$ implies $U \rhd V$.
\item
Suppose $U$ is sequential and, for all $V$, we have $U \rhd_{\sf loc} V$ implies $U \rhd V$.
Then, $U$ is reflexive. 
\item
Finitely axiomatized sequential theories are not reflexive. (See \cite{pudl:cuts85}.)
\end{enumerate}

\section{Basic definitions and insights}\label{badein}

In this section we give a precise definition of Ali's question.
My main interest is in Ali's question for sequential theories. However, it is good to have a wider definition.
Setting things up with a bit more generality will enable us to apply some informal rigour
to the choice of notions. We will see that already for Vaught theories, Ali's question makes good sense.

 We consider broader and narrower versions of the
question. We will show that the question takes its most natural form in the sequential case.

\subsection{Basic formulation of the question}
We first address the treatment of numerals. What is needed here is that the theories we consider interpret
at least some minimal theory successor. For this we choose the theory ${\sf Succ}_0$.
The theory ${\sf Succ}_0$ has one constant 0 and one binary predicate {\sf S}. I will use infix notation for {\sf S}.
The axioms of ${\sf Succ}_0$ are as follows. We define $\widetilde 0(x) := (x=0)$ and
$\widetilde{(n+1)}(x) := \exists y\, (\widetilde{n}(y) \wedge y\Su x)$. 
\begin{enumerate}[${\sf Succ}_0$1.]
\item $\exists x \, \widetilde n(x)$
\item $\neg\, (\widetilde m(x) \wedge \widetilde n(x))$, where $m < n$
\item $(\widetilde n(x) \wedge \widetilde n(y)) \to x=y$
\end{enumerate} 
So, the theory just says that successor behaves normally as long as we are finitely far removed from
0. We have the following small insight.

\begin{theorem}\label{babysmurf}
${\sf Succ}_0$ is a sub-theory of {\sf R}, where we identify $x\Su y$ with ${\sf S}x=y$. 
\end{theorem}

\noindent
Consider any theory $U$ and suppose $N:  {\sf Succ}_0\to U$. 
Suppose $N$ is $n$-ary. We expand the signature of $U$ with a new $n$-ary predicate {\sf T}.
We write ${\sf T}(\underline n)$ for: 
\[\exists  \vec x \in \delta_N\, (\widetilde n^N(\vec x) \wedge  {\sf T}(\vec x)).\]
(In case we want to emphasize the dependence of our numerals on $N$ we write ${\sf T}(\underline n^N)$.)

The class of sentential Tarski biconditionals ${\sf TB}^{-}_N$ consists of the sentences of the form $A \iff {\sf T}(\gn{A})$.
Here $A$ is a $U$-sentence and $\gnum{A}$ is the G\"odel number of $A$. We will usually omit the underlining and simply
write ${\sf T}(\gnum{A})$.

We say that \emph{$N$ has the Enayat Property} or that \emph{$U$ is an $N$-Enayat theory}
iff $U\rhd (U+{\sf TB}^{-}_N)$, or, i.o.w., $\top\rhd_U {\sf TB}^-_N$.

We note that $N$ is part of the data for $\underline n$ and
 $U$ is part of the data for $N$ and the signature of $U$ is part of the data for $U$.
So, indeed, the notation `${\sf TB}^{-}_N$' exhibits all the necessary data, with the exception of the G\"odel numbering,
 to construct the intended set of sentences. In most of the paper, we will treat the G\"odel numbering as fixed in the back-ground,
 where the G\"odel numbering is supposed to be a standard efficient G\"odel numbering.
  Exceptions are  Subsection~\ref{smurfeastwood} and Section~\ref{blijesmurf}. In Subsection~\ref{blijesmurf}, we need a non-standard numbering.
  In Subsection~\ref{smurfeastwood}, we show that, in the Vaught case, the Enayat property does not depend on the G\"odel numbering under
  the appropriate assumptions of G\"odel numberings. 

\begin{remark}
We note that, even if our framework is fairly general, the theory ${\sf Succ}_0$ may be still too restrictive.
The point is that nothing really seems to depend on the uniqueness of the numerals as stipulated in ${\sf Succ}_03$. 
The only advantage of the present approach is that we can use the numeral notation in a meaningful way. 
One important  advantage of the more general approach, where we drop ${\sf Succ}_03$, is that
also pair theories are covered by the framework. 
\end{remark}

\noindent
Here is a first small observation. This observation is well-known. I do not know who first made it.

\begin{theorem}\label{gretigesmurf}
Suppose $N:U \rhd {\sf Succ}_0$. Then, $\top \jump_{{\sf loc},U}{\sf TB}^-_N$. 
\end{theorem}

\begin{proof}
We interpret the axioms $A_0 \iff {\sf T}(\gnum{A_0})$, \dots, $A_{n-1} \iff {\sf T}(\gnum{A_{n-1}})$, by
defining:
${\sf T}(\vec x) :\iff \bigvee_{k<n} (\widetilde{\gnum{A_k}}(\vec x) \wedge A_k)$.
\end{proof}

\subsection{Intensionality}
Enayatness, as defined here, is an \emph{intensional} property, since, in the general case,
 it critically depends both on $N$ and on the choice of the 
G\"odel numbering.  In Section~\ref{blijesmurf}, we will see an example that illustrates these
dependencies.

We will see that, in the case that $U$ is a Vaught theory, Enayatness is independent of the G\"odel numbering,
assuming that all G\"odel numberings that we allow are recursively related to some standard G\"odel numbering.
Secondly, we will see  that, for Vaught theories, Enayatness can be considered as a 
property of theories rather than of interpretations. 

In the case of sequential theories, we
can even do better: we can give a characterization of Enayatness in which G\"odel numerings nor truth are mentioned!

\subsection{Enayatness for Vaught theories}\label{smurfeastwood}
We show that, for Vaught theories, the property of Enayatness is independent of the choice of numerals.
Moreover, we show that, for Vaught theories, Enayatness is independent of the choice of
the G\"odel numbering, as long as the G\"odel numberings considered are recursively related to a standard one.


We have the following theorem:
\begin{theorem}\label{moppersmurf}
Suppose $U$ is a Vaught theory and $N,M:{\sf Succ}_0 \lhd U$.
Then, ${\sf TB}^-_N\mutfuj_U{\sf TB}^-_M$. 
It follows that $M$ has the Enayat Property iff $N$ does.
\end{theorem}

\begin{proof}
We use a minor adaptation of a well-known argument due to Dedekind and Pudl\'ak.
See \cite{pudl:cuts85}.

We define $\mathfrak F  := \mathfrak F_{N,M}$ beween $\delta_{N}$ and $\delta_{M}$ as follows.
$\vec x\mathrel{\mathfrak F}\vec y$ iff there a partial bijection $f$ between $\delta_{N}$ and $\delta_{M}$
such that (i) $f0_{N} = f0_{M}$, (ii) if $\vec x \Su_N \vec y$ and $f\vec y$ is defined, then
$f \vec x$ is defined and $f\vec x \Su_{M} f\vec y$. The definition of partial bijection is provided via the
direct interpretation of {\sf VS} in $U$.

One now easily shows that
\[\mathfrak F(\underline n^N) =_M \underline n^M \text{ and }\forall \vec y\; ( (\delta_M(\vec y) \wedge 
\mathfrak F(\underline n^N) =_M \vec y) \to \vec y =_M \underline n^M).\]

\noindent
We  interpret $U +{\sf TB}^{-}_M$ via the translation, say $\tau$, in $U+{\sf TB}^{-}_N$ by taking the identical translation for the
$U$-vocabulary and setting \[{\sf T}_\tau (\vec y) := \exists \vec x\, ( \delta_N(\vec x) \wedge  \vec x \mathrel \mathfrak F \vec y \wedge {\sf T}(\vec y)).\] 
The interpretation of $U +{\sf TB}^{-}_N$ in $U+{\sf TB}^{-}_M$ is similar.
\end{proof}

\noindent
We have found that, if a Vaught theory has the Enayat property for some $N$, it has the
Enayat property for all $N$. If a Vaught theory has the Enayat property for some $N$, we
will call it simply an  \emph{Enayat theory}. 

 Here is a convenient observation. 
 \begin{theorem}\label{jockeysmurf}
 Suppose $U$ is a Vaught theory. Then $U$ is Enayat iff $U$ interprets $U+{\sf TB}^{-}_N$, for some $N:{\sf R}\lhd U$.
 Here the dimension of $N$ can be taken to be the dimension of the direct interpretation that establishes Vaughtness.
 \end{theorem}
 
 \begin{proof}
 Consider a Vaught theory $U$.
 By Theorem~\ref{motorsmurf}, $U$ interprets {\sf R}, say, via $N$, where  $N$ is the composition of
 the one-dimensional interpretation of {\sf R} in {\sf VS} and the direct interpretation of {\sf VS} in $U$.
 It follows that the dimension of  $N$ is the dimension of this direct interpretation. 
 
Since {\sf R} extends ${\sf Succ}_0$, by Theorem~\ref{babysmurf}, $U$ interprets ${\sf Succ}_0$
via the interpretation $N'$ based on $\tau_N$. 
Clearly, $U+{\sf TB}^-_N$ is extensionally the same as $U+{\sf TB}^-_{N'}$. So, we are immediately done.
\end{proof}

\noindent
When considering Vaught theories, we will from now on consider interpretations of {\sf R}.

We address the worry that Enayatness for Vaught theories may be crucially dependent on
details of the chosen G\"odel numbering. 

\begin{theorem}
Suppose we have G\"odel numberings $\nu_0$ and $\nu_1$. We only need to assume that the $\nu_i$ assign numbers to sentences.
Suppose for some recursive function $\eta$ we have $\nu_0 = \eta \circ \nu_1$. Suppose $N:{\sf R}\to U$.
Then, ${\sf TB}^{-\nu_0}_N \jump_U {\sf TB}^{-\nu_0}_N$. 

It follows that if $U$ is a Vaught theory that is Enayat for $\nu_0$, then $U$ is
Enayat for $\nu_1$.
\end{theorem}

\begin{proof}
We assume the conditions of the theorem. 
In {\sf R}, we can represent the function $\eta$ by a formula $H$. We now define a Fujimoto translation $\tau$ as follows.
\[{\sf T}^\tau(\vec x) := \delta_N(\vec x) \wedge \exists \vec y\in \delta_N\, (H^N(\vec x,\vec y\,) \wedge {\sf T}_0(\vec y\,)).\] 
It is easy to see that $\tau$ delivers the goods. 
\end{proof}

\noindent
We note that we need not impose any \emph{a priori} restriction on the complexity of $\nu_1$ for the theorem to work. Of course,
our default assumption is that we are working with a reasonable G\"odel numbering.

Here is the general form of our conjecture for Vaught theories.

\begin{conjecture}\label{smulsmurf}
No finitely axiomatized consistent Vaught  theory is Enayat.
\end{conjecture}

\noindent
We also can ask a more modest question.
\begin{question}\label{voetbalsmurf}
 Suppose there is a finitely axiomatized, consistent Vaught theory that is Enayat. Can we show, under that assumption, that \emph{all}
finitely axiomatized, consistent, Vaught theories are Enayat theories?
\end{question}

\noindent
We end this subsection with one further question.

\begin{question}\label{hockeysmurf}
In Subsection~\ref{goochelsmurf}, we will show that in the recursively enumerable sequential case, 
we can characterize Enayat theories in a coordinate-free
way. Not only is the question of Enayatness independent of the G\"odel numbering, but G\"odel numberings are not mentioned in the
characterization. Can we do something similar in the Vaught case?
\end{question}

\subsection{Enayatness for sequential theories}\label{goochelsmurf}
In case our theories are sequential and recursively enumerable, 
we can do better than the previous section by eliminating any reference to numerals and coding from the definition of
Enayat theory. 

We start with a convenient observation. 
 \begin{theorem}\label{sprintermurf}
 Suppose $U$ is a sequential. Then $U$ is Enayat iff $U \rhd (U+{\sf TB}^{-}_N)$ for some $N:{\sf S}^1_2\lhd U$.
 We can take the dimension of $N$ to be the dimension of the direct interpretation of {\sf AS} in $U$ that establishes
 sequentiality.
 \end{theorem}

\begin{proof}
The proof is entirely analogous to the proof of Theorem~\ref{jockeysmurf}, noting that sequential theories are Vaught and
that ${\sf S}^1_2$ extends {\sf R} and that the interpretation of ${\sf S}^1_2$ in {\sf AS} is one-dimensional.
\end{proof}

\noindent
We say that an interpretation $K:U\lhd V$ is \emph{sententially restricted} if, there is an $n$, such that, for all 
$U$-sentences $B$, there is a $V$-sentence $C$, such that $\rho(C) \leq n$
and $V \vdash B^K \iff C$. Here $\rho$ is the complexity measure \emph{depth-of-quantifier-alternations}.
See \cite{viss:smal18} for a careful treatment of the measure.

Suppose $N:{\sf S}^1_2 \lhd V$.
 We say that an interpretation $K:U\lhd V$ is \emph{strongly sententially restricted} w.r.t. $N$ iff, for some $V$-formula $A(\vec x\,)$,
 where the length of $\vec x$ is the dimension of $N$, we have that, for all
$U$-sentences $B$, $V \vdash B^K \iff A(\gnum{B}^N)$.

\begin{theorem}
Suppose $K:U\lhd V$, where $V$ is sequential and recursively enumerable. Let $N:{\sf S}^1_2 \lhd V$. Then, $K$ is sententially restricted iff
$K$ is strongly sententially restricted \textup(w.r.t. $N$\textup).
\end{theorem}  

\begin{proof}
The right-to-left direction is immediate, noting that numerals only contribute a constant to the complexity independent of the
size of the numeral. 

We treat left-to-right. Suppose $K:U \lhd V$ is sententially restricted. Let the witnessing number be $n$.
Let $\gamma$ be the the function that takes as input a $U$-sentence $B$, searches for the smallest (coded) $V$-proof
with conclusion of the form $B^K \iff C$, where $\rho(C) \leq n$, and gives as output $C$. Clearly $\gamma$ is a total recursive function.

 Let ${\sf True}_n(\vec x)$ be a truth-predicate for $V$-sentences of complexity $\leq n$ based on a 
satisfaction predicate ${\sf Sat}_n(\vec s,\vec x)$ where we can prove the commutation clauses for formulas of complexity $\leq n$ that are
in a suitable $V$-provable cut $J$ of $N$. The length of the sequence $\vec x$ is the dimension of $N$.
We note that, since standard numbers are in $J$, we have the 
Tarski biconditionals for ${\sf True}_n$ for sentences of the right complexity. 
See \cite{viss:smal18} for a detailed treatment of partial truth predicates in sequential theories.

We define $A(\vec x) := {\sf True}_n(G^N(\vec x))$, where  $G$ stands for the representation in the arithmetical
language of the recursive function $\gamma$.
Note that we really should have written $\exists \vec y \in \delta_N\, (G^N(\vec x,\vec y) \wedge {\sf True}_n(\vec y))$. 

Consider any $U$-sentence $B$. Suppose $\gamma(B)= C$. We have:
\qedright
\begin{eqnarray*}
V \vdash A(\gnum{B}) & \iff & {\sf True}_n(G^N(\gnum{B})) \\
 & \iff  & {\sf True}_n(\gnum{C}) \\
& \iff & C \\
& \iff & B^K  
\end{eqnarray*}
\end{proof}

\noindent
We now have immediately the following consequence.

\begin{theorem}
Let $V$ be sequential. Then, $V$ is an Enayat theory iff $V$ has a sententially restricted self-interpretation.
\end{theorem}

\noindent
We note that, for decidable theories, the identity interpretation is restricted. So,
having a restricted self-interpretation generally is much broader than being Enayat.

\noindent
We may now formulate the following conjecture:
\begin{conjecture}\label{hapjessmurf}
No finitely axiomatized consistent sequential  theory is Enayat.
Equivalently, no finitely axiomatized consistent sequential  theory has a sententially restricted
self-interpretation.
\end{conjecture}

\subsection{Preservation over mutual interpretability}

Surprisingly, Enayatness is preserved for over mutual interpretability if we take
{\sf R} rather than ${\sf Succ}_0$ as the basic arithmetical theory that provides the numerals.
More precisely we have the following.

\begin{theorem}
Suppose $K:U \rhd V$, $M:V \rhd U$, $N:V \rhd {\sf R}$, $P: V \rhd (V+{\sf TB}^{-V}_{N})$. 
Here the superscript $V$ is there to remind us that we consider {\sf TB} for the singature of $V$. 
Let $\mathcal E: V \to (V+{\sf TB}^{-V}_{N})$ be the identical embedding. Let $N^\ast: = K\circ P \circ \mathcal E \circ N$.
Then, $U$ is an $N^\ast$-Enayat theory.
\end{theorem}

\begin{proof}
Let $Q_0 := K\circ P \circ \mathcal E \circ M$. So, $Q_0 : U\rhd U$. More graphically, the situation looks like this
(using the category theoretical notation for interpretations):

 \[   U \stackarrow{M}  V \stackarrow{\mathcal E} (V+{\sf TB}^{-V}_{N})\stackarrow{P} V
 \stackarrow{K} U \] 
 
 \noindent
We extend $Q_0$ to $Q:U \rhd (U + {\sf TB}^{-U}_{N^\ast})$
as follows: $\tau_Q$ is $\tau_{Q_0}$ on the vocabulary of $U$.
\begin{itemize}
\item 
${\sf T}_{\tau_{\sf Q}}(\vec x) :\iff {\sf T}_{K\circ P}(\widetilde{\tau}_M(\vec x))$.
\end{itemize}
 Here $\widetilde \tau_M$ is the arithmetization of $\tau_M$ in $N^\ast$. 
  We have:
 \qedright
 \begin{eqnarray*}
 U \vdash {\sf T}^Q (\gnum A^{N^\ast}) & \iff & {\sf T}^{K\circ P}(\widetilde{\tau}_M(\gnum{A}^{N^\ast})) \\
 & \iff & {\sf T}^{K\circ P}(\gnum{A^M}^{N^\ast}) \\
 & \iff & ({\sf T} (\gnum{A^M}^N))^{K\circ P}\\
 & \iff & (A^M)^{K\circ P} \\
 & \iff &  A^{ K \circ P \circ \mathcal E \circ M} \\
 & \iff & A^{Q_0}\\
 & \iff & A^{Q}
 \end{eqnarray*}
 \end{proof}

\noindent
We have the following corollary:

\begin{corollary}\label{gurusmurf}
Suppose $U$ and $V$ are mutually interpretable Vaught theories. Suppose further that $V$ is an
Enayat theory. Then, $U$ is an Enayat theory.
\end{corollary}

\begin{question}\label{tennissmurf}
Are there any other (interesting) relations between theories that preserve Enayatness?
\end{question}

\noindent
In case we consider sequential theories, Theorem~\ref{gurusmurf} has an important consequence.
We remind the reader that every finitely axiomatized sequential $A$ is mutually interpretable
with ${\sf S}^1_2+{\sf con}_{\rho(A)}(A)$.
So, our question about examples of sequential Enayat theories reduces to the question whether
${\sf S}^1_2$ plus a true $\Pi^0_1$-sentence can be Enayat.

We can strengthen our question as follows.

\begin{conjecture}\label{extrasmurf}
Suppose $A$ is finitely axiomatized and consistent and sequential. Let $N:{\sf S}^1_2\lhd A$.
Then, there is no extension of ${\sf S}^1_2$ that is mutually interpretable with $A+{\sf TB}^{-}_N$.

It is a well known open question whether every sequential theory is mutually interpretable
with an extension-in-the-same-language of ${\sf S}^1_2$. Our conjecture provides a possible example
to illustrate a negative answer to this question.
\end{conjecture}

\noindent
We can view the preservation over mutual interpretability a bit more abstractly in the case of Vaught theories.
Suppose we work with the degrees of interpretability of Vaught theories.
In this case the Tarski functor (based on) $\mathfrak T(U) := U+ {\sf TB}^{-}_N(U)$, where $N$ is some
$N: {\sf R} \lhd U$, makes sense, since we have already shown the independence of $N$. 

We check that $\mathfrak T$ is indeed a functor. Suppose $K:U \lhd V$. Then, we can extend $\tau_K$ to, say, $\widetilde\tau_K$ 
as follows. Suppose $N :{\sf R} \lhd U$.
We choose {\sf T} over $V$ w.r.t. the $NK$-numerals. 
We extend $\tau_K$ to $\tau^\ast_K$ by setting ${\sf T}_{\tau^\ast_K}(\vec x) :\iff {\sf T}(\widetilde \tau_K(\vec x))$,
where $\widetilde \tau_K$ is the formalization of $\tau_K$.

We note that an Enayat Vaught theory is precisely a $\mathfrak T$-algebra. So, it is immediate that being an Enayat theory is
preserved under mutual interpretability which is after all the isomorphism of our category.

Perhaps it is possible to make $\mathfrak T$ work for a better category, but I did not explore this. 

Here is an alternative formulation of conjecture~\ref{smulsmurf}.
Let's say that a theory $U$ is \emph{quasi-finite} iff, for some finitely axiomatized $A$,
we have $U \equiv A$.

\begin{conjecture}\label{megasmurf}
Suppose $U$ is  a consistent Vaught theory. Then $\mathfrak T(U)$ is not quasi-finite.
\end{conjecture}

\noindent To see the equivalence, we present the following consideration.
Clearly, if there were a consistent and Vaught Enayat theory $A$, then $\mathfrak T(A)$ would be quasi-finite.

Conversely, suppose $U$ is consistent and Vaught. Suppose  $\mathfrak T(U)\equiv A$.
 There is a finitely axiomatized sub-theory $V_0$ of $\mathfrak T(U)$ such that $V_0 \rhd A$.
We choose $V_0$ large enough so that it is a Vaught theory. Let $U_0$ consist of the $U$-axioms in $V_0$.
We note that $U_0 \rhd V_0$, since we can interpret finitely many Tarski biconditionals for free, by Theorem~\ref{gretigesmurf}.
Thus, $U_0 \rhd V_0 \rhd A \rhd \mathfrak T(U) \rhd \mathfrak T(U_0)$. We may conclude that $U_0$ is a finitely axiomatized, consistent
and Vaught Enayat theory. 

\section{A consistent, finitely axiomatized Enayat Theory}\label{blijesmurf}
Is there a finitely axiomatized theory with the $N$-Enayat property for appropriate $N$? If we ask the question in this
generality without further constraints on the admissible theories, there
is actually a positive example. The example does depend on what we accept as a G\"odel numbering. We
discuss the issues here below.

We give an example of a finitely axiomatized theory that is not Enayat for one interpretation (and for any G\"odel numbering) and
that is Enayat for another interpretation for a special choice of the G\"odel numbering.
 
We consider the  theory $W:= {\sf Th}_{0,{\sf S},<}(\mathbb N)$ of $0$, $<$ and {\sf S} in the natural numbers. See \cite[Section 3.2]{endemath01}
for a careful exposition of this theory.
The theory $W$ is a finitely axiomatizable complete theory, to wit, the theory of a discrete linear ordering with initial and without final point.
Every definable set of numbers in the language of $W$ over $\mathbb N$ is either finite or cofinite. Moreover, inspection of the quantifier elimination
shows that the theory has a multi-exponential decision algorithm.\footnote{One further amazing property of $ {\sf Th}_{0,{\sf S},<}(\mathbb N)$
is the fact that it is a finitely axiomatizable theory that proves full induction.}

Suppose $\imath$ is the direct one-dimensional translation of ${\sf Succ}_0$ in $W$ that sends 0 to 0 and $x{\sf S}y$ to
${\sf S}x = y$. Then, clearly, the interpretation $K_\imath$ of ${\sf Succ}_0$ in $W$ based on $\imath$ cannot be Enayat
for any G\"odel numbering since the set of truths is infinite and co-infinite. We cannot get around this example
by tweaking the G\"odel numbering.

Let $W':= {\sf Th}_{0,{\sf S},<,{\sf E}}(\mathbb N)$ be the theory of $0$, $<$, {\sf S} and {\sf E}, for \emph{even}, in the natural numbers
Let $\jmath$ be the following two-dimensional translation of the language of $W'$ in the language of $W$. 
\begin{itemize}
\item
$\delta_\jmath(x,y) := (y=x \vee y = x+1)$,
\item
 ${\sf Z}_\jmath(x,y) := (x=0 \wedge y = 0)$,
 \item
 $(x,y) {\sf S}_\jmath(x',y') :=  ((y = x \wedge x'= x \wedge  y' = x+1)\; \vee$\\
 \hspace*{2.7cm} $(y=x+1 \wedge x' =x+1 \wedge y'=x+1))$,
 \item 
 $(x,y) <_{\jmath} (x',y') := ((x' = x \wedge y<y') \vee x< x')$.
 \item
 ${\sf E}_\jmath(x,y) := (x=y)$.
 \end{itemize}
  
  \noindent
Clearly, this yields an interpretation $K_\jmath$ of $W'$ in $W$ based on $\jmath$.
We note that it follows that $W'$ is multi-exponentially decidable.

 Let $\nu$ be a standard G\"odel numbering for the language of $W$. We define $\nu^\ast(A) := 2 \nu(A)$ if $A$ is true in $\mathbb N$ and
$\nu^\ast(A) = 2\nu(A)+1$ if $A$ is false. Evidently, $\nu^\ast$ is a multi-exponential G\"odel numbering.
Let $\kappa$  translate {\sf T} to {\sf E} where $\kappa$ is the identical translation on the vocabulary of $W$.
Clearly, $K_\kappa: (W + {\sf TB}^{-\nu^\ast}_{\sf ID}) \lhd W'$. Hence, $(K_\kappa\circ K_\jmath): (W + {\sf TB}^{-\nu^\ast}_{\sf ID}) \lhd W$.

Let $N$ be the interpretation of $W$ in $W$ based on $\jmath$ restricted to the language without {\sf E}.
We  have $\top  \jump_W {\sf TB}_N^{-\nu^\ast}$, showing that Tarski's Theorem on the undefinability of
truth fails in our example for a specific choice of G\"odel numbering and a specific choice of the numbers.
Of course, there is nothing remarkable about this failure, since we do not have the Fixed Point Lemma in this context.

We note that we can do the same trick for, e.g., Presburger Arithmetic. However, Presburger is not finitely axiomatizable.

The reader may object that our G\"odel numbering $\nu^\ast$ is contrived, unnatural and an ignoble hack. However, it seems very difficult
to exclude it on principled reasons. One may want to demand that G\"odel numberings are p-time. However, many of the classical
G\"odel numberings were exponential or even multi-exponential. This is witnessed by, e.g., the G\"odel numbering in Feferman's celebrated 
arithmetization paper \cite{fefe:arit58}.

\begin{question}\label{ruimtesmurf}
Is there an example of a finitely axiomatized theory $A$ with the $N$-Enayat property for some $N:{\sf Succ}_0\lhd A$, when we demand that the
G\"odel numbering is p-time computable?
\end{question}

\begin{question}\label{rugbysmurf}
Is there an example of a finitely axiomatized theory $A$ such that we have the Enayat property for all $N:{\sf Succ}_0\lhd A$?
\end{question}

\section{Neighbours}\label{neighbours}
In this section, we discuss uniform variants of ${\sf TB}^-$. We will see that
for the uniform variants we have a clear negative answer ---quite unlike the stubborn purely sentential case of ${\sf TB}^-$.

\begin{remark}
When writing this paper I discovered that much more can be said about uniform biconditionals and Vaught theories.
I postpone this to a subsequent paper. 
\end{remark}

\noindent
We fix a theory $U$ with an interpretation $N:U \rhd {\sf Succ}_0$.
In order to avoid heavy and sometimes misleading notations, we assume
$N$ to be one-dimensional. Nothing depends on this however.

\subsection{Satisfaction}
We strengthen ${\sf TB}_N^{-}$ to a uniform principle ${\sf USB}_{1,N}^{-}$ in the following way. 
\begin{description}
\item[${\sf USB}^-_{1,N}$]
$ \forall  x\,   (\sat(x,\gnum{A(v)}) \iff A(x))$. 
\end{description}
Here $A$ is a $U$-formula with at most one free variable $v$ and \sat\ is a new binary predicate.
Note that this definition is meaningful also in case our theory is not Vaught.

In case $U$ is Vaught, we also have the following seemingly stronger principle.
\begin{description}
\item[${\sf USB}^-_{N}$]
$ \forall  a\,   (\sat(a,\gnum{A(v_0,\dots,v_{n-1})}) \iff A(a(\gnum{v_0}),\dots, a(\gnum{v_{n-1}})))$. 
\end{description}
Here $a$ ranges over assignments, i.e., partial functions from a finite set of variables to domain objects.
If $v_i$ is not in the domain, we set value of the variable to some default value $x^\ast$.
Regrettably, in the general case, $x^\ast$ must be a parameter, since there need not be
definable elements in the ambient theory $U$.

\begin{theorem}\label{extravagantesmurf}
Suppose $U$ is a Vaught theory and $N:U\rhd {\sf R}$. Then, we have ${\sf USB}_{1,N}^- \mutfuj_U {\sf USB}_N^-$.
\end{theorem}

\begin{proof}
If we start with $U+{\sf USB}^-_N$, we can use Fujimoto translation $\tau$
with:
\begin{itemize}
\item
 $\sat^\tau(x,y) := \sat(\verz{\tupel{\gnum{v},x}},y)$.
 \end{itemize}
 We note that we pretended that we have functionality. In reality, we should have said
 that there is a representative $u$ of $\gnum{v}$, a representative $w$ of the pair 
 $\tupel{u,x}$, a representative $z$ of the set $\verz{w}$, such that $\sat(z,y)$.
 Moreover, what counts as a representative of $\tupel{u,x}$, should also be further spelled out.
 
 Suppose we start with $U+{\sf USB}^-_{1,N}$. We define a recursive function $F$ that sends
 a code of a formula $A(v_0,\dots,v_{n-1})$ to a code of $A(v(\gnum{v_0}),\dots,v(\gnum{v_{n-1}}))$.
 Here of course the functions should be unraveled to their relational representations.
 \emph{Par abus de langage}, we use $F$ also for the arithmetization of $F$ in ${\sf R}$.
 We now use the following Fujimoto translation $\nu$:
 \begin{itemize}
\item
 $\sat^\nu(a,y) := \sat(a, F^N(y))$.
 \end{itemize}
Again, unraveling is needed to give the formula its correct form: 
\qedright
\[\exists z\in \delta_N\, (F^N(y,z) \wedge \sat(a,z)).\]
\end{proof}

\begin{theorem}\label{krabbiesmurf}
If $U$ is Vaught and $N$, $N'$ are interpretations in $U$ of ${\sf Succ}_0$ \textup(or, if you wish, {\sf R}\textup),
then ${\sf USB}^-_{1,N}$ and ${\sf USB}^-_{1,N'}$ and ${\sf USB}^-_N$ and ${\sf USB}^-_{N'}$ are mutually Fujimoto interpretable over $U$.
\end{theorem}

\noindent The proof is entirely analogous to the proof of Theorem~\ref{moppersmurf}.
We also have:
\begin{theorem}\label{smoothysmurf}
$\top \jump_{{\sf loc},U} {\sf USB}^-_{1,N}$ and $\top \jump_{{\sf loc},U} {\sf USB}^-_N$.
\end{theorem}

\begin{proof}
We treat the case of $ {\sf USB}^-_{1,N}$.
We interpret the axioms 
\[A_0(y) \iff \sat(y,\gnum{A_0}),\; \dots,\; A_{n-1}(y) \iff \sat( y,\gnum{A_{n-1}}),\]
 by defining:
$\sat(y,x) :\iff \bigvee_{k<n} (\widetilde{\gnum{A_k}}(x) \wedge A_k(y))$.
\end{proof}

\noindent
Here is a basic insight.
\begin{theorem}\label{papanijnsmurf}
Suppose $N: U \rhd {\sf Succ}_0$.
Then,  $\mho(U) \rhd (U+{\sf USB}_N^-)$, and, similarly for
 ${\sf USB}^-_{1,N}$.
\end{theorem}

\begin{proof}[Proof sketch]
In $\mho(U)$ we can build a Henkin interpretation $H$ of $U$. (See \cite{viss:inte18}.) This Henkin interpretation
comes with a satisfaction predicate {\sf H} that works on a $\mho(U)$-cut $I$. Since $\mho(U)$ is sequential,
 there is a definable isomorphism $\mathfrak F$ between a cut of $I$ and a cut of $NH$.
We take $\sat(x,y) := \exists x'\,\exists y'\, (x'\mathfrak F x \wedge y'\mathfrak F y \wedge {\sf H}(x',y'))$. 
\end{proof}

\noindent 
We discuss two alternative forms of $ {\sf USB}^-_{N}$.
Let us write ${\sf comm}(\sat,x)$ for: for all formulas $\leq x$, the
predicate \sat\ satisfies the commutation conditions (w.r.t. the signature of $U$).
Suppose $U$ is a Vaught theory and that $N:U \rhd {\sf R}$.
We define:
\begin{description}
\item[${\sf Comm}^-_{0,N}$]
${\sf comm}(\sat,\underline n^N)$, for $n\in \omega$.
\end{description}

\noindent
 As is well known the theory {\sf R} interprets an extension, say ${\sf R}^+$,
which verifies that $\leq$ is a linear ordering. See \cite{viss:whyR14}.
Suppose $U$ is a Vaught theory and that $N:U \rhd {\sf R}^+$.
\begin{description}
\item[${\sf Comm}^-_{1,N}$]
$\forall x \in  \mathfrak J\; {\sf comm}(\sat,x)$, $\forall x\in \mathfrak J\,\forall y \leq x\; y\in \mathfrak J$,
$\underline n\in \mathfrak J$, for any $n\in \omega$. Here $\mathfrak J$ is a new unary predicate.
\end{description}

\noindent
We assume that our G\"odel coding is monotonic in the sense the the code of a subformula of $A$
is less that the code of $A$ itself.

\begin{theorem}\label{smakkiesmurf}
Suppose $U$ is a Vaught theory.
We have:
\begin{enumerate}[a.]
\item
If $N:U \rhd {\sf R}$, then ${\sf USB}^-_N$ and ${\sf Comm}^-_{0,N}$ are interderivable over $U$.
\item
If $N:U \rhd {\sf R}^+$, then
$ {\sf Comm}^-_{1,N} \vdash_U {\sf Comm}^-_{0,N}$ and 
${\sf Comm}^-_{0,N}\jump_U{\sf Comm}^-_{1,N}$.
\end{enumerate}
\end{theorem}

\begin{proof}
\emph{Ad \textup(a\textup):} The inference from ${\sf Comm}^-_{0,N}$ to ${\sf USB}_N$ is obvious.
We treat the case of existential quantification for the other direction. We reason in $U+{\sf USB}_N$.
We write $\breve a$ for $a[\gnum{v_i}:v_i]$, the result of resetting $a$ on $\gnum{v_i}$ to $v_i$.
\begin{eqnarray*}
\sat(a, \exists v_i\, A(v_0,\dots,v_i,\dots,v_{n-1})) & \iff & \exists v_i\, A(a(\gnum{v_0}),\dots, v_i,\dots,a(\gnum{v_{n-1}})) \\
& \iff & \exists v_i\, A(\breve a (\gnum{v_0}),\dots, \breve a(\gnum{v_i}),\dots,\breve a(\gnum{v_{n-1}})) \\
& \iff & \exists v_i\, \sat(\breve a, \gnum{A(v_0,\dots, v_i,\dots,v_{n-1})})
\end{eqnarray*}

\noindent
\emph{Ad \textup(b\textup):}
This is mostly trivial. We interpret $\mathfrak J(x)$ as ${\sf comm}(\sat,x)$.
\end{proof}

\noindent 
The difference between ${\sf Comm}^-_{0,N}$ and ${\sf Comm}^-_{1,N}$ may seem 
somewhat trifling, but the usefulness of ${\sf Comm}^-_{1,N}$ lies in the fact that
there may be other more interesting interpretations of $\mathfrak J$.

We note that  ${\sf Comm}^-_{0,N}$ is a restricted axiomatization of ${\sf USB}_N^-$ over $U$.
This means that all axioms of ${\sf Comm}^-_{0,N}$ have depth-of-quantifier-alternations complexity
below a fixed $n$. 
This suggests the following question.

\begin{question}\label{golfsmurf}
Does ${\sf TB}_N^-$ have a restricted axiomatization over $U$?  
\end{question}

\subsection{Truth}
There is also the alternative option of defining a truth principle.

\begin{description}
\item[${\sf UTB}^-_N$]
$ \forall \vec x \in \delta_N\, ({\sf T}\gnum{A\dot{\vec x}\,} \iff A{\vec x})$.\\
\end{description}

\noindent
We note that to make sense of this we must stipulate that (i) $N$ is an interpretation of
${\sf S}^1_2$ and  that (ii) we use efficient numerals, since, for ordinary numerals, the mapping from
$x$ to the numeral of $x$ is exponential. 

It is not clear to me that, in this case, we have an analogue of Theorem~\ref{krabbiesmurf}, i.e., that 
${\sf UTB}^-_N \mutfuj_U{\sf UTB}^-_{N'}$. However, as we will see, it is immediate from
Theorem~\ref{stevigesmurf}, that these theories are
mutually interpretable in the sequential case.

Here is a first small insight.

\begin{theorem}\label{betjesmurf}
Let $U$ be sequential and $N:U \rhd {\sf S}^1_2$. Then,
 ${\sf USB}^-_N\jump_U {\sf UTB}_N^-$. 
\end{theorem}

\begin{proof}[Proof sketch]
We have to define  the predicate {\sf T} from \sat. Consider a number $a$. In case $a$ is not an $N$-code of a $U$-sentence
we make ${\sf T}(a)$ false. Suppose $U$ is an $N$-code of a $U$-sentence. Now we have to analyze $a$ as being  a substitution instance
of a $U$-formula $b$ with numerals. There are two obstacles:
{\footnotesize
\begin{itemize}
\item
There are not really numerals in $b$, since we work with a relational signature.
So, we have to reverse the term-unwinding translation to obtain the relevant numerals. 
To do this we need a precise analysis of term-unwinding. Also, we should take care that the reverse algorithm is
 p-time.  
\item
In the $A(\dot{\vec x})$ of ${\sf UTB}^-_N$ there could be already numerals in the standard context
$A(\cdot)$. However $U$ having just $a$ as input cannot know which numerals are the numerals to replace
by variables. Fortunately, it is sufficient to remove numerals maximally. The case where there are some numerals
 in $A(\cdot)$ can be recovered by substituting some numerals in the result of maximal analysis.
\end{itemize}
}
\noindent
Given that we analyzed $a$ as substitution instance of $b$ where we replace numeral $c$ by variable $v$,
we can compute a corrresponding assignment $f$ that sends $v$ to the value of $c$.
Now we  define ${\sf T}(a)$ by $\sat(f,b)$.  
\end{proof}

\noindent
We do not generally have that ${\sf TB}_N^- \jump_U {\sf UTB_N^-}$.

\begin{theorem}
${\sf TB}_{\sf ID}^- \not \jump_{\sf EA} {\sf UTB_{\sf ID}^-}$.
\end{theorem}

\begin{proof}
Suppose ${\sf TB}_{\sf ID}^- \jump_{\sf EA} {\sf UTB_{\sf ID}^-}$. 
Then, it follows that ${\sf TB}_{\sf ID} \jump_{\sf PA} {\sf UTB_{\sf ID}}$.
Here the lack of the superscript minus means that we extend induction to
the full language. However, there is a model of ${\sf PA}+{\sf TB}_{\sf ID}^-$ that is
not recursively saturated, where all models of  ${\sf PA}+{\sf UTB}_{\sf ID}^-$
are recursively saturated. See \cite{cies:epis17}.
\end{proof}

\begin{question}\label{damsmurf}
The argument above works for all subtheories of {\sf PA} that extend {\sf R} and more, but still it is rather
special. Can we improve it to show that the result holds for all sequential theories?
\end{question}

\subsection{$\mho$}
We now connect uniform biconditionals with the $\mho$-functor.

\begin{theorem}\label{nijntjesmurf}
Suppose $U$ is sequential and $N:U\rhd {\sf S}^1_2$. Then $(U+{\sf UTB}^-_N)\rhd \mho(U)$.
\end{theorem}

\begin{proof}[Proof sketch]
In $U+{\sf UTB}^-_N$, we can define the intersection $\mathcal J$ of all virtual classes
$\verz{ x\in \delta_N \mid {\sf T}(a(\dot x))}$ that are $N$-cuts. One can show that $\mathcal J$ is an $N$-cut contained in
all $U$-definable cuts. Thus, in $\mathcal J$,  we have all restricted consistency statements of $U$.
\end{proof}

\noindent
We partially summarize the above in the following theorem.
\begin{theorem}\label{stevigesmurf}
Suppose $U$ is sequential and $N:U\rhd {\sf S}^1_2$. Then, following theories are mutually interpretable:
$\mho(U)$, $U+{\sf USB}^-_N$, $U+{\sf USB}^-_{0,N}$, $U+{\sf Comm}^-_{0,N}$, $U+{\sf Comm}^-_{1,N}$, $U+{\sf UTB}^-_N$.

As a consequence, a sequential theory is \emph{uniformly Enayat}, in any of the possible senses,
iff it is reflexive.
\end{theorem}

\begin{proof}
We have:
\begin{itemize}
\item $\mho(U) \rhd (U+{\sf USB}^-_N)$, by Theorem~\ref{papanijnsmurf};
\item 
${\sf USB}^-_{N} \mutint_U {\sf USB}^-_{0,N}$,  by Theorem~\ref{extravagantesmurf};
\item
$U+{\sf Comm}^-_{0,N}$, $U+{\sf Comm}^-_{1,N}$ and $U+{\sf USB}^-_{N}$ are mutually interpretable by Theorem~\ref{smakkiesmurf};  
\item
$(U+{\sf USB}^-_N)\rhd (U+{\sf UTB}^-_N)$, by Theorem~\ref{betjesmurf};
\item
$(U+{\sf UTB}^-_N) \rhd \mho(U)$, by Theorem~\ref{nijntjesmurf}.
\end{itemize}
The last step completes the circle.
\end{proof}

\noindent
We note that theories 
like {\sf PRA} and {\sf PA} and {\sf ZF} are reflexive and, hence,  sequential uniform Enayat theories.

We also note that the characterization of $\mho(U)$ as $U+{\sf USB}_N^-$ (modulo mutual interpretability)
has the advantage of having the G\"odel numbering as conventional element, but not the proof system, the
arithmetization of the proof system and the like.

The following corollary is immediate.

\begin{corollary}
No finitely axiomatized, consistent uniform sequential theory is uniformly Enayat.
\end{corollary}

\noindent
So, for example, none of ${\sf S}^1_2$, ${\sf EA}$, ${\sf ACA}_0$, {\sf GB} is uniformly Enayat.

\begin{question}\label{schaaksmurf}
Can we prove the non-existence of a finitely axiomatized consistent sequential uniform Enayat theory
 without a detour over the second incompleteness theorem?
\end{question}

\noindent
More can be said about {\sf USB} in the context of Vaught theories. We hope to do that in a subsequent paper.
A salient open question is as follows.

\begin{question}\label{gosmurf}
Is there a finitely axiomatized  Vaught theory that is uniformly Enayat? Here uniformity is explicated
using ${\sf USB}^-$.
\end{question}

\section{Finite extensions of ${\sf TB}^{-}$}
In this section, we formulate two conjectures in the environment of Conjecture~\ref{hapjessmurf}.

Consider a theory $U$ of signature $\Theta_0$. Let $\Theta_1$ be $\Theta_0$ extended with a unary predicate {\sf T} and let
$\Theta_2$ be  binary predicate symbol \sat.
The variables $\alpha,\beta, \dots$ range over sentences of $\Theta_2$.
We take as the default that a theory has as signature the minimal signature demanded by its axioms.

In this vocabulary, we can state Tarski's theorem on the undefinability of truth as follows.
\begin{theorem}[Tarski]
Suppose $U$ is consistent and $N:U\rhd {\sf R}$. Then, we have
 $\top \not\jump_U {\sf TB}^-_N$.
 \end{theorem}
 
 \noindent
 From Tarski's work on truth we also know the following.
 
 \begin{theorem}[Tarski]
 Suppose $U$ is a Vaught theory and $N:U \rhd {\sf R}$.
 Then,
 $(\forall x\in \delta_N\, {\sf comm}(\sat,x)) \jump_U {\sf TB}^-_N$.
 \end{theorem}
 
 \noindent
Of course, this is just a watered down version of Theorem~\ref{smakkiesmurf}(a).

In the next theorem, we show that there is no `best', in the sense of `weakest', finite extension of a
Vaught theory $U$ in an extended signature that Fujimoto interprets
${\sf TB}^-_N$ over $U$. So, certainly the commutation conditions, as articulated by $\forall x\in \delta_N\, {\sf comm}(\sat,x)$,
are not `best'.
 
 \begin{theorem}\label{slaapsmurf}
 Suppose $U$ is a consistent Vaught theory and  $N:U \rhd {\sf R}$.
 Suppose $\alpha\jump_U {\sf TB}^-_N$.
 Then, there is a $\beta$ with $\beta \jump_U  {\sf USB}^-_N$, but
 $\beta\not\jump_U\alpha$.
 
 It follows that for $\gamma := (\alpha\vee \beta)$, we have $ {\sf TB}^-_N \jumpbneq_U\gamma \jumpbneq_U \alpha$.
 Moreover, in case $\alpha\jump_U {\sf USB}^-_N$, we find $ {\sf USB}^-_N \jumpbneq_U\gamma \jumpbneq_U \alpha$.
 \end{theorem}
 
 \noindent The proof is a variation of proof of Theorem~4.1, case (A),  of \cite{viss:ques19}.
 
 \begin{proof}
 Suppose $A$ is of the form $\exists x\in \delta_N\, A_0(x)$.
 We write ${\sf C}(A)$, 
 for `if there is no witness $x$ of $A$ such that $x\leq y$, then ${\sf comm}(\sat,y)$'.
 In other words, ${\sf C}(A)$ is $\neg\, ((\exists y\in\delta_N \, \neg\, {\sf comm}(\sat,y)) \leq A)$.
 
 We write $\delta \jump_U \eta$ as $\exists p\, \exists \tau\, {\sf proof}_{U+\delta}(p,\eta^\tau)$. 
  By the Fixed Point Lemma, we find $B$ such that ${\sf R} \vdash B\; \iff\; {\sf C}(B^N) \jump_U  \alpha$.
  We take $\beta := {\sf C}(B^N)$.
 
 Suppose $\beta\jump_U\alpha$. Then, we find that $B$ is true and, hence, $B$ has a standard witness inside
 $U,N$. However, finitely many commutation conditions are Fujimoto-interpretable in $U$, by the combination of
 Theorems~\ref{smoothysmurf} and \ref{smakkiesmurf}(a).
 Hence, we have $\top \jump_U \beta \jump_U \alpha \jump_U {\sf TB}_N^-$, contradicting Tarski's Theorem
 on the undefinability of truth. So, $\beta\not\jump_U\alpha$.
 
 It follows that $B^N$ has no standard witness inside $U+\beta$. So,
 $U+\beta$ provides the commutation conditions at all standardly finite levels, i.e., $\beta\jump_U {\sf Comm}_{0,N}$.
 Hence, as desired, 
 $\beta \jump_U{\sf USB}_N^-$. 
 \end{proof}
 
 \begin{remark}
 We note that, in the proof of Theorem~\ref{slaapsmurf}, the sentence $B$ is a fixed point of a formula 
 of essentially the form `provable \dots'. This might convey the impression that we have G\"odelean self-reference here.
 However, the internal ${\sf C}(B)$ has Rosser-form. Thus, it seems very improbable that $B$ is uniquely determined by
 the equation even if our numbers satisfy ${\sf S}^1_2$. For the same reason, an `explicit' solution for $B$ seems improbable.  
 \end{remark}

\begin{conjecture}\label{partysmurf} Let $U$ be Vaught.
and let $N:{\sf R}\lhd U$. Suppose $\alpha \jump_U{\sf TB}^{-}_N$. Then,
 $\alpha \jump_U{\sf USB}^-_{N}$.

We can put further demands on $U$ and $N$: that $U$ be sequential, finitely axiomatized, etcetera; that $N$ is an interpretation of ${\sf S}^1_2$, etcetera.
Also, when $N$ interprets ${\sf S}^1_2$, we may replace ${\sf USB}^-_N$ by ${\sf UTB}^-_N$.
\end{conjecture}

\noindent
We note that the Tarski commutation conditions, are an example of such an $\alpha$. We note that the construction in the proof of
Theorem~\ref{slaapsmurf} does not immediately help to refute the conjecture. However, it cannot be excluded that some variant of the argument
does refute the conjecture.

\begin{remark}
Conjecture~\ref{partysmurf} can be connected to the Davidsonian idea that
we need compositionality to obtain a finite axiomatization of the Tarski Biconditionals.
\end{remark}

\begin{remark}
The only paper I could find asking a question in the neighbourhood of Conjecture~\ref{partysmurf} is \cite{fine:trut84}.
However, Fine and McCarthy work with what they call \emph{segregated languages}. Their format does not seem
to fit ours. Moreover, they do not work with Fujimoto interpretability.
It would be well worth exploring what of their work can be adapted to our context.
\end{remark}

\begin{remark}
Conjecture~\ref{partysmurf} suggests the concept of Fujimoto preservativity $\rightslice$.
We define:
\begin{itemize}
\item
$\Gamma \rightslice_U \Delta$ iff, for all $\alpha$ such that $\alpha \jump_U \Gamma$, we have
$\alpha \jump_U \Delta$.
\end{itemize}
{\footnotesize
Here, the most elegant approach is to take $\alpha$, $\Gamma$ and $\Delta$ to be in the language of $U$ expanded with
a binary predicate. We also want to apply the notion if one of the expansions is to unary, but we can choose some
standard way to let a binary predicate pose as a unary one, e.g., we might take $Rxx$ for $Px$.
Alternatively, in a context where we have pairing, we might only consider expansions with a unary predicate.
A final alternative is to assume that intended signatures for $\alpha$, $\Gamma$ and $\Delta$ are given in the context.
}

Now Conjecture~\ref{partysmurf} becomes: we have ${\sf TB}_N^- \rightslice_U {\sf UTB}^-_N$.

We note that $\Gamma \jump_U \Delta$ implies  $\Gamma \rightslice_U \Delta$. Moreover, $\rightslice_U$ is reflexive and transitive.
\end{remark}

\noindent
Here is our second conjecture.

\begin{conjecture}\label{yogasmurf}
Suppose $A$ is a finitely axiomatized Vaught in signature $\Theta_0$. Let $N:A\rhd {\sf R}$. Suppose further that
 $\top \rhd_A {\sf TB}^-_N$. Then, there is a  $\beta$  such that
$\top\rhd_A \beta \jump_A {\sf TB}_N^-$.  

More generally, we may conjecture the following. Suppose $A$ is a finitely axiomatized Vaught theory and $\top \rhd_A V$.
Then, there is a $B$ such that $\top\rhd_A B \jump_A V$.
\end{conjecture}

\noindent
We have:
\begin{theorem}
The truth of Conjectures~\ref{partysmurf} and \ref{yogasmurf} combined implies the truth of Conjecture~\ref{hapjessmurf}.
\end{theorem}

\begin{proof}
Suppose $\top \rhd_A {\sf TB}^{-}_N$. Let $\beta$ be as promised in Conjecture~\ref{yogasmurf}, so we have
$\top \rhd_A \beta \jump_A{\sf TB}_N^-$. 
It follows, by Conjecture~\ref{partysmurf}, that $\top \rhd_A \beta \jump_A{\sf USB}^-_{N}$. However, we have already seen that,
in the sequential case, $\top\nrhd_A {\sf USB}^-_{N}$.
\end{proof}

\section{A provability predicate}\label{propre}
In the present section we follow the Saccheri strategy. We assume that we have a finitely axiomatized, sequential $A$
that is Enayat. We pretend that it is consistent and explore it as an interesting new world.

\subsection{Preliminary considerations}
The first thing that one thinks of is the Liar Paradox for the ${\sf TB}^{-}$-truth predicate.
More precisely:
suppose we have an interpretation $K:A \rhd (A+{\sf TB}^{-}_{N_0})$, where $N_0:A \rhd {\sf S}^1_2$. 
Let $\mathfrak F:= \mathfrak F_{N,NK}$ be the usual isomorphism between initial cuts of $N_0$ and $N_0K$.
We define ${\sf K}(\vec x) :\iff \exists \vec y\, (\vec x\mathrel{\mathfrak F}\vec y \wedge {\sf T}^K(\vec y))$.
We write ${\sf K}B$ for ${\sf K}(\gnum{B}^{N_0})$. By the G\"odel Fixed Point Lemma, 
we find $L$ such that $A\vdash L \iff \neg\, {\sf K}L$. We can see that the truth value of
$L$ has to alternate if we travel inside $K$, $K^2$, $K^3$, \dots. Also, we have
$({\sf ID}_A\tupel{L}K): A \rhd (A+L)$ and $(K\tupel{L}{\sf ID}_A): A \rhd (A+\neg \,L)$. However, nothing paradoxical follows.\footnote{We pick
up the idea of using an analogue of a semantical paradox in Section~\ref{para}.}

We can see that without some further idea nothing paradoxical \emph{can} follow, since
if $U$ is e.g. ${\sf EA}+\verz{{\sf con}^n({\sf EA})\mid n\in \omega}$ we do have
that (i) $U$ is a restricted theory in the sense that the complexity of all its
axioms is bounded by a fixed $n$ and (ii) $U$ is reflexive for the identical interpretation and, so, $U \rhd (U+{\sf TB}^{-}_{{\sf ID}_U})$.

Thus, we need to add an ingredient that essentially uses the fact the $A$ is finitely axiomatized rather than just restricted.
In this section, this ingredient is the use of a new provability predicate for $A$, the good properties of 
which are based on $A$'s finite axiomatizability. In Section~\ref{para}, the ingredient is the use of a conjecture
that is supposed to hold only for finitely axiomatizable sequential theories.

\subsection{What we fix}
In this section, we consider a number of things as fixed: 
\begin{enumerate}[i.]
\item
the theory $A$ itself;
\item
the interpretation $\mathcal S$ that witnesses the sequentiality of $A$;
we note that, by cut-elimination, the proof of
${\sf AS}^{\mathcal S}$ can be taken to have complexity ${\sf max}(\rho(A),\rho(\mathcal S))+\mathfrak a_0$;
here $\mathfrak a_0$ is a constant for overhead.
\item
the interpretation $N_0$ of ${\sf S}^1_2$;
we note that, by cut-elimination, the proof of
$({\sf S}^1_2)^{N_0}$ can be taken to have complexity ${\sf max}(\rho(A),\rho(N_0))+\mathfrak a_1$;
we note that, since there is a known interpretation of ${\sf S}^1_2$ in {\sf AS}, there is an $N_0$ of complexity
$\rho(\mathcal S)+ \mathfrak a_2$;
\item
the interpretation $K$ of $A+{\sf TB}^{-}_{N_0}$ in $A$.
\end{enumerate}

\noindent
We note that the complexity of ${\sf K}(\underline n)$ is $\rho(K)+\rho(N_0)+ \mathfrak a_3$, where
$\mathfrak a_3$ is a constant for overhead.

\subsection{The provability predicate}
Let  $n\geq \mathfrak a^\ast := {\sf max}(\rho(A),\rho(K)+\rho(N_0))+\mathfrak a$, where $\mathfrak a$ is 
suitable constant number that is needed for the overhead in our reasoning. 
We note that in the context of our reasoning $\mathfrak a^\ast$ functions as a constant since we treat
$A$, $K$ and $N_0$ as fixed.

Our new provability predicate
is $\apr_{A,(n)} B := \opr_{A,n}{\sf K}B$. We note that $n \geq \mathfrak a^\ast$ is precisely what is needed to make $\apr_{A,(n)}$
a meaningful notion. We use subscript $(n)$ rather that $n$ to remind the reader that the subscript is \emph{not} a constraint on
the sentences for which our predicate is meaningful. 

\emph{We will omit the subscript $A$ since $A$ is fixed in this section and we will have
to work with a whole circus of sub- and superscripts anyway.}

We remind the reader of the L\"ob conditions. For a brief moment, we view $\apr$ as an abstract operator.
\begin{enumerate}[{\sf L}1.]
\item
$\vdash B \;\; \To \;\; \vdash \apr B$,
\item
$\vdash (\apr B \wedge \apr (B\to C) ) \to \apr C$,
\item
$\vdash \apr B \to \apr\apr B$,
\item
$\vdash \apr(\apr B \to B) \to \apr B$.
\end{enumerate}
We will also consider the condition {\sf C}, to wit:
\begin{enumerate}[{\sf C}.]
\item
$\vdash \apr B \to B$.
\end{enumerate}
In the next subsection, we discuss variants of {\sf L}1 and {\sf L}2 for the predicate $\apr_{(n)}^N$, where $N$ is a cut of $N_0$.

\subsection{The first two L\"ob conditions}\label{onetwo}
Let $N$ be an $A$-definable, $A$-verifiable cut of $N_0$. We note that $N$ has the same numerals as $N_0$.
Let $n\geq\mathfrak a^\ast$.
 
\begin{theorem}\label{loebunus}
We have: 
\begin{enumerate}[A.]
\item
if $A \vdash B$, then ${\sf S}^1_2 \vdash \apr_{(n)} B$, and, hence $A \vdash \apr_{(n)}^NB$. Thus, 
this gives us  ${\sf L}1$ in the form:
\[ A\vdash B \;\; \To \;\;  A \vdash \apr_{(n)}^N B.\]
 \item
 We also have a second form that is a better `externalisation' of ${\sf L}3$.
\[ A \vdash_n {\sf K}B \;\; \To \;\; A \vdash_n {\sf K} \apr^N_{(n)} B.\]  
\end{enumerate}
\end{theorem}

\begin{proof}
Ad (A): Suppose $A \vdash B$. Then, $A \vdash B^K$ and, hence, 
$A \vdash {\sf K}B$. By cut-elimination, we have $A \vdash_n {\sf K}B$. By $\Sigma_1$-completeness, in the meta-theory, 
we find ${\sf S}^1_2 \vdash  \opr_{n} {\sf K}B$, i.o.w., ${\sf S}^1_2 \vdash  \apr_{(n)} B$.

Ad (B): Suppose $A \vdash_n {\sf K}B$. Then, ${\sf S}^1_2 \vdash \apr_{(n)} B$.
It follows that $A \vdash \apr^{NK}_{(n)}B$ and, hence, $A\vdash {\sf K}\apr^N_{(n)}B$.
By cut-elimination, we find  $A\vdash_n {\sf K}\apr^N_{(n)}B$.
\end{proof}

\noindent 
As the reader can see, $n$ is, in this theorem, not constrained by $N$. The reason is that, since $N$ is initial in $N_0$, the $N$-numerals
simply are the $N_0$-numerals. 

We note that in both proofs we used cut-elimination. Thus, the proofs can be executed in meta-theory
${\sf EA}^+$, i.e., $\mathrm I\Delta_0 + {\sf supexp}$.
So, we have:
\begin{itemize}
\item ${\sf EA}^+ \vdash \opr B \to \opr \apr^N_{(n)} B$,
\item
${\sf EA}^+ \vdash \apr_{(n)} B \to \apr_{(n)}\apr^N_{(n)} B$.
\end{itemize}
 However, on closer inspection, we have a much better result. We suppose that  $A$, $n$, $N$ and $K$ are
externally given. Then, the first form only requires ${\sf S}^1_2$ and the second form only requires ${\sf EA} := \mathrm I\Delta_0+{\sf exp}$. We 
will discuss this in detail when we consider proofs of {\sf L}3 in the next subsection.

\begin{theorem}\label{loebduo}
We have ${\sf S}^1_2 \vdash (\apr_A B \wedge \apr_A(B\to C)) \to \apr C$. So, \emph{a fortiori}, 
the theory $A$ satisfies ${\sf L}2$ for  $\apr_A^N$.
\end{theorem}

\begin{proof}
We have: 
\begin{eqnarray*}
A \vdash {\sf K}(B\to C) & \to & (B\to C)^K \\
& \to & (B^K \to C^K) \\
& \to & ({\sf K}B \to {\sf K}C).
\end{eqnarray*}  
By cut-elimination, $A \vdash_n {\sf K}(B\to C) \to ({\sf K}B \to {\sf K}C)$. It follows by $\Sigma_1$-completeness that
${\sf S}^1_2 \vdash \opr_{n}({\sf K}(B\to C) \to ({\sf K}B \to {\sf K}C))$. 
By {\sf L2} for $\opr_{n}$, we find:
\[{\sf S}^1_2\vdash \opr_{n}{\sf K}(B\to C) \to (\opr_{n}{\sf K}B \to \opr_{n}{\sf K}C).\]
In other words, ${\sf S}^1_2\vdash (\apr B \wedge \apr (B\to C) ) \to \apr C$.
\end{proof}

\noindent
We note that there does not seem to be a way to prove the uniform version
 \[ \lightning\lightning\;\;\; {\sf S}^1_2 \vdash \forall B,C\, ((\apr_A B \wedge \apr_A(B\to C)) \to \apr C).\;\; \lightning\lightning\]
The quantifier over sentences seems essentially external.

\begin{remark}\label{smurfeleen}
With the first  L\"ob condition in hand, we can immediately prove the well-known properties of the G\"odel sentences.
Let $N$ and $n$ be as before. By the G\"odel Fixed Point Lemma, we find $G_{N,n} := G$ such that ${\sf S}^1_2 \vdash G \iff \neg \,\apr_{(n)} G^N$.
Suppose $A \vdash G^N$. Then, by {\sf L}1, $A \vdash \apr^N_{(n)}G^N$. On the other hand, by the Fixed Point Equation,
 $A \vdash \neg\, \apr^N_{(n)}G^N$. It follows that $A \vdash \bot$.
 
 We cannot similarly derive L\"ob's Rule as will be illustrated in Subsection~\ref{custossmurf}. The derivation of the Rule does need
 some form of {\sf L}3.
 \end{remark}
 
 \subsection{Guarded reflection}\label{custossmurf}
We have the following theorem:

\begin{theorem}
For any $n$, there is  an $A$-definable, $A$-verifiable $N_0$-cut 
$\mathfrak I_n$ with $\rho$-complexity of order $\mathfrak b n+ {\sf max}(\rho(\mathcal S),\rho(N_0)) + \mathfrak c$, such that, for all $B$ with $\rho(B)\leq n$,
we have $A \vdash \opr^{\mathfrak I_n}_{n} B \to B$. 
\end{theorem}

\noindent 
See \cite{viss:smal18}, for a careful treatment of this result (or, rather, a result of which this result is an immediate consequence).
We now have a form of guarded reflection for $\apr$. 

\begin{theorem}\label{flexismurf}
Let $n \geq \mathfrak a^\ast$.
Then,  we have $A \vdash \apr_{(n)}^{\mathfrak I_n} B \to B^K$. 
\end{theorem}

\noindent As preparation for Remark~\ref{smurfeltwee}, we formulate a well-know lemma that is due to
Pudl\'ak. See \cite{pudl:cuts85}.

\begin{lemma}\label{lekkerbeksmurf}
Suppose $N,N':A \rhd {\sf S}^1_2$. There is an $N$-cut ${\mathfrak C}_{N,N'}$ and a definable isomorphic embedding 
$\mathfrak F_{N,N'}:\mathfrak C_{N,N'}\to N'$. We have: 
{\small
 \[\rho(\mathfrak C_{N,N'}) = {\sf max}(\rho(\mathcal S),\rho(N),\rho(N'))+ \mathfrak d \text{ and }
\rho(\mathfrak F_{N,N'}) = {\sf max}(\rho(\mathcal S),\rho(N),\rho(N'))+ \mathfrak e.\]}
\end{lemma} 

\noindent We write $\aco$ for $\neg\apr\neg$.

\begin{remark}\label{smurfeltwee}
We have $A \vdash \aco_{(n)}^{\mathfrak I_n}\top$. So, we immediately see that the Second Completeness Theorem fails.
Hence, \emph{a fortiori}, L\"ob's Rule fails.

If we did not have the guard $K$, we would have $A$'s inconsistency. Let $N$ be a cut of $\mathfrak I_n$.
Then, we have $A \vdash \apr^N_{(n)} G^N_{N,n} \to G^N_{N,n}$. Hence, $A \vdash G^N_{N,n}$ and, so $A \vdash \apr_{(n)}^NG^N_{N,n}$,
which gives $A \vdash \neg\, G^N_{N,n}$. So $A\vdash \bot$. 

One thing one could try, in order to get the effect of the above reasoning, is to get under the guard using Lemma~\ref{lekkerbeksmurf}.
We have a brief look, to see why this idea fails.
Suppose $N$ is a cut both of $\mathfrak I_n$ and $\mathfrak C_{N_0,N_0K}$.

We first try $G_{N,n}$. We reason in $A$.
Suppose $\apr_{(n)}^N G^{N}_{N,n}$. Then, by guarded reflection, $G^{NK}_{N,n}$.
But also $\apr_{(n)}^{N_0K} G^N_{N,n}$. Hence, $\neg\, G^{N_0K}_{N,n}$. However,
since $NK$ is smaller than $N_0K$, no contradiction materializes.

Next we try $G_{N_0,n}$. We reason in $A$.
Suppose $\apr_{(n)}^N G^{N_0}_{N_0,n}$. Then, by guarded reflection, $G^{N_0K}_{N_0,n}$.
But also $\apr_{(n)}^{N_0K} G^{N_0}_{N_0,n}$. Hence, $\neg\, G^{N_0K}_{N_0,n}$. 
A contradiction. So, canceling the assumption, we find $\neg\,\apr_{(n)}^N G^{N_0}_{N_0,n}$, i.e. $G^N_{N_0,n}$.
Returning to the meta-language, we see that $A \vdash G^{N}_{N_0,n}$. Of course,
this is still no contradiction.

The result of these two attempts is somewhat disappointing. However, we 
will see in Subsection~\ref{mangosmurf} that a G\"odel-style argument does
give us some information about $K$.
\end{remark}

\subsection{The third and fourth L\"ob condition}
We write $\graysq$ for provability in ${\sf S}^1_2$.
We define ${\sf itexp}(0,x) := x$ and ${\sf itexp}(y+1,x) := 2^{{\sf itexp}(y,x)}$.

We start with a lemma.

\begin{lemma}\label{smartsmurf}
Suppose $n \geq {\sf max}(\rho(A),\rho(B))$.
We have ${\sf S}^1_2 \vdash \opr B \to \graysq \opr_n B$.
\end{lemma}

\begin{proof}
We use the version of cut-elimination from Buss' paper \cite{buss:situ15}.
By formalizing Buss' result, we have, ${\sf S}^1_2$-verifiably, that whenever
$p$ is a proof and whenever ${\sf itexp}(\rho(p)+2,p)$ exists, then we have a cut-free proof
$q$ with the same conclusion.

Secondly, we use an insight from Pudl\'ak's paper \cite{pudl:cuts85},
that, \[{\sf S}^1_2 \vdash \exists w \, 2^y = w \to \graysq \forall x\, \exists z\, {\sf itexp}(y,x) = z.\]

\noindent
We reason as follows inside ${\sf S}^1_2$. Suppose $p$ is an $A$-proof of $B$.
Then, $\rho(p)$ is a logarithmic number. So $\graysq \exists z\,{\sf itexp}(\rho(p)+2,p) = z$.
We also find $\graysq \mathsf{proof}_A(p,B)$. Hence,
inside $\graysq$ we have a cut-free proof $q$ that witnesses
$\opr_n B$.
\end{proof}

\noindent
Let $N$ and $N'$ be 
$A$-definable, $A$-verifiable cuts of $N_0$ and let $n\geq \mathfrak a^\ast$.
We first internalize Theorem~\ref{loebunus}(A).

\begin{theorem}\label{wijzesmurf}
We have ${\sf S}^1_2 \vdash \opr B \to \graysq \apr_{(n)}B$, and, hence,
\[ {\sf S}^1_2 \vdash \opr B \to \opr \apr^{N}_{(n)}B  \text{ and } A \vdash  \opr^{N'} B \to \opr^{N'} \apr^{N}_{(n)}B.\]
\end{theorem}

\begin{proof}
The proof is an internalization and refinement of the proof of Theorem~\ref{loebunus}(A).
We note that externally we have $A \vdash B^K \iff {\sf K}B$, and, thus, 
${\sf S}^1_2 \vdash \opr (B^K \iff {\sf K}B)$. It follows that (\dag)
${\sf S}^1_2 \vdash \opr B^K \iff \opr {\sf K}B$.

We reason in ${\sf S}^1_2$.
Suppose $\opr B$. Then, since $A$ is finitely axiomatized and interpretations give p-time 
transformations of proofs, we have $\opr  B^K$. Hence, by
(\dag), $\opr {\sf K}B$. By Lemma~\ref{smartsmurf}, it follows that $\graysq \opr_n{\sf K}B$, i.o.w., $\graysq\apr_{(n)}B$.
\end{proof}

\noindent
We proceed with an internalization of Theorem~\ref{loebunus}(B). Let $N$ and $N'$ be as before.
Let $n\geq \mathfrak a^\ast$. We use a slightly more general formulation with an extra $m$ for later use.

\begin{theorem}\label{scom}
Suppose $S$ is in $\Sigma_1^{\sf b}$ and  $m\geq \rho(S)+\rho(K) + \rho(N) + \mathfrak f_0$, where $\mathfrak f_0$ 
is a constant for overhead. We have
 ${\sf S}^1_2 \vdash S \to \apr_{(m)}S^N$. Hence, we have
$A\vdash S^{N'} \to \apr_{(m)}^{N'}S^N$.
\end{theorem}

\begin{proof}
We have $A \vdash S^{NK} \to {\sf K}S^N$. So, by cut-elimination,
 $A \vdash_m S^{NK} \to {\sf K}S^N$. 
 Note that this makes sense only under our assumption on $m$.
 By $\Sigma_1$-completeness, we find:
   ${\sf S}^1_2 \vdash  \opr_{m}(S^{NK}  \to {\sf K}S^N)$.
   Hence,  (\dag) ${\sf S}^1_2 \vdash  \opr_{m}S^{NK}_{n}  \to \apr_{(m)}S^N$.

We reason in ${\sf S}^1_2$.
Suppose $S$.  By $\exists\Sigma_1^{\sf b}$-completeness,
we have
$\opr_{m}S^{NK}$. By applying (\dag), we find $\apr_{(m)}S^N$.
\end{proof}

\noindent By specializing Theorem~\ref{scom}, we find:

\begin{theorem}\label{loebtres}
Suppose  $m\geq \rho(K) + \rho(N) + \mathfrak f$, where $\mathfrak f$ 
is a constant for overhead.
We have ${\sf S}^1_2 \vdash \apr_{(n)} B \to \apr_{(m)}\apr_{(n)}^N B$. Hence, we have
$A\vdash \apr_{(n)}^{N'}B \to \apr_{(m)}^{N'}\apr_{(n)}^N B$.
\end{theorem}

\noindent
Now, putting $m=n$,  {\sf L}4 follows in the usual way from {\sf L}1,2,3 in combination with the Fixed Point Lemma.

\begin{theorem}\label{loebquattro}
Under the assumption that  $n \geq \rho(K) + \rho(N) + \mathfrak f$, we have:
\[ {\sf S}^1_2 \vdash \apr_{(n)}(\apr_{(n)}^N B \to B) \to \apr_{(n)} B.\]
Hence, we have:
\[ A \vdash  \apr^{N'}_{(n)}(\apr_{(n)}^N B \to B) \to \apr^{N'}_{(n)} B.\]
\end{theorem}

\noindent
Assuming $A$ to be consistent, we note that, for any $n$ and $N$ satisfying the assumption of Theorem~\ref{loebquattro}, we cannot have
guarded reflection for $\apr^N_{(n)}$. Otherwise, we would have both L\"ob's Theorem and  guarded reflection at the same time. 
But this is impossible since it would follow that $A \vdash \aco_{N,(n)}\top$.

We can give an alternative form of {\sf L}3 where we eliminate the lower bound on $n$ at the cost of relativizing the antecedent
to a cut.

\begin{theorem}\label{lepesmurf}
\begin{enumerate}[1.]
\item
Consider any $m$.
There is an ${\sf S}^1_2$-cut $\mathfrak J_m$, such that $\rho(\mathfrak J_m)$ is of order  
$\mathfrak gn+\mathfrak h$ and
$ {\sf S}^1_2 \vdash \forall x \in \mathfrak J_m\,
\exists w\, {\sf itexp}(m,x) = w$.
\item
Suppose $m,n \geq \mathfrak a^\ast$. Then, ${\sf S}^1_2 \vdash \apr_{(m)}^{\mathfrak J_{m+2}}B \to \apr_{(n)} B$.
\end{enumerate}
\end{theorem}

\begin{proof}
Ad (1):
The proof is essentially contained in \cite{pudl:prime83} or \cite{pudl:cuts85}.
Given any cut $I$, we consider the virtual class $\verz{x \mid 2^x \in I}$. This class is
downward closed and closed under successor. We shorten it to a cut $J$.
Inspecting the construction, clearly, $\rho(J) = \rho(I) +\mathfrak g$.
We obtain $\mathfrak J_m$ by iterating the construction starting from the identical cut.

Ad (2): We reason in ${\sf S}^1_2$. Suppose $\apr_{(m)}^{\mathfrak J_{m+2}}B$.
This means $\opr_{m}^{\mathfrak J_{m+2}}{\sf K}B$. Let the witnessing proof be $p$.
Since $p\in \mathfrak J_{m+2}$, we find, by (1), that ${\sf itexp}(m+2,p)$ exists.
So, by the cut-elimation theorem from  \cite{buss:situ15}, we find
$\opr_{\mathfrak a^\ast}{\sf K}B$, so, \emph{a fortiori}, $\opr_{n}{\sf K}B$, i.e.,
$\apr_{(n)}B$. 
\end{proof}

\noindent
It is always good to have an alternative proof of a result. In the proof of Theorem~\ref{lepesmurf}, we 
used  ${\sf S}^1_2$-formalization of Buss' result of  \cite{buss:situ15}, a delicate result that involves many
details. So, it improves our confidence to have a variant of Theorem~\ref{lepesmurf}(2), with a different proof.

We write $\widetilde{\sf K}$ for $\neg{\sf K}\neg$.
Let $\mathfrak H_{C}: ({\sf S}^1_2+\oco_{\mathfrak a^\ast} \widetilde{\sf K}C) \rhd (A+\widetilde{\sf K}C)$ 
be the Henkin interpretation based on $\oco_{\mathfrak a^\ast} \widetilde{\sf K}C$. 
See \cite{viss:inte18}. We note that  $\mathfrak H_{C}$ is uniform in $C$. We
can view `$C$' as a variable.
Let $\mathfrak D_{m}$ be the common cut in ${\sf S}^1_2$ of the identical interpretation and all the
$\mathfrak I_m \mathfrak H_{C}$, for $C$ such that $\oco_{\mathfrak a^\ast} \widetilde{\sf K}C$.
We note that the complexity of $\mathfrak H_{C}$ is a small standard number independent of $C$. 
So, $\rho(\mathfrak D_{m})$ is
$\rho(\mathfrak I_m)$ plus some standard constant.

\begin{theorem}
Suppose $m,n \geq \mathfrak a^\ast$. Then, ${\sf S}^1_2 \vdash \apr_{(m)}^{\mathfrak D_{m}} B \to \apr_{(n)} B$.
\end{theorem}

\begin{proof}
We reason in ${\sf S}^1_2$. Suppose $\oco_{n}\widetilde{\sf K}C$. Then, \emph{a fortiori}, we have
 $\oco_{\mathfrak a^\ast}\widetilde{\sf K}C$ and, hence,
 $(A+\widetilde{\sf K}C)^{\mathfrak H_{C}}$.
So, we have, by reasoning inside $\mathfrak H_{C}$, that  $\oco_{(m)}^{\mathfrak I_m\mathfrak H_{C}}\widetilde{\sf K}C$.
Hence, by the definition of $\mathfrak D_m$, we find $\oco_{(m)}^{\mathfrak D_{m}}\widetilde{\sf K}C$.

We return  to the meta-language. We note that $A \vdash {\sf K}\neg\, B \iff \neg\, {\sf K}B$, and, hence
$A \vdash_n {\sf K}\neg\, B \iff \neg\, {\sf K}B$ and $A \vdash_m {\sf K}\neg\, B \iff \neg\, {\sf K}B$.
So, ${\sf S}^1_2 \vdash \opr_n{\sf K}\neg \, B \iff \opr_n\neg\, {\sf K} B$ and, similarly, for $m$. Hence, putting $C := \neg\, B$,
we obtain our desired result.
\end{proof}

\noindent Let $S$ be again $\Sigma_1^{\sf b}$, let $N$  again be a cut of $N_0$ and let again $n\geq \mathfrak a^\ast$.
Let $m^\ast :=  \rho(S)+\rho(K) + \rho(N) + \mathfrak f_0$.

\begin{theorem}\label{scomtwee}
We have ${\sf S}^1_2 \vdash S^{\mathfrak J_{m^\ast+2}} \to \apr_{(n)}S^N$.
\end{theorem}

\begin{proof}
We reason in ${\sf S}^1_2$. Suppose $S^{\mathfrak J_{m^\ast+2}}$. By Theorem~\ref{scom},
we have $\apr_{(m^\ast)}^{\mathfrak J_{m^\ast+2}}S^N$.
 By Theorem~\ref{lepesmurf}, it follows that $\apr_{n}S^N$.
 \end{proof}

\noindent
By specializing we find the following.
Let $m^\ast :=  \rho(K) + \rho(N) + \mathfrak f$.

\begin{theorem}
We have ${\sf S}^1_2 \vdash \apr_{(n)}^{\mathfrak J_{m^\ast+2}}B \to \apr_{(n)}\apr_{(n)}^NB$.
\end{theorem}

\noindent
We have the following extension of Theorem~\ref{scomtwee}. This version could play a role in the development
of Rosser arguments.
Let $S$ be  $\Sigma^0_1$, let $N$   be a cut of $N_0$ and let  $n\geq \mathfrak a^\ast$.
Let $m^\ast :=  \rho(S)+\rho(K) + \rho(N) + \mathfrak f_1$.
\begin{theorem}\label{scomdrie}
We have ${\sf S}^1_2 \vdash S^{\mathfrak J_{m^\ast+4}} \to \apr_{(n)}S^N$.
\end{theorem}

\begin{proof}
This uses the well-known fact that ${\sf S}^1_2 \vdash S^{\mathfrak J_2} \to \graysq_{\rho(S)+\mathfrak i} S$.
See \cite{haje:meta91}. Here $\mathfrak i$ s a small constant for overhead.
\end{proof}

\begin{remark}\label{smurfeldrie}
Suppose $N$ is shorter than $\mathfrak {I}_n$. Let $G := G_{N,n}$. We reason in $A$.
Suppose $\apr^{\mathfrak J_nN}_{(n)}G^N$. Then, since $\mathfrak J_nN$ is initial in $N$,
we have $\apr^{N}_{(n)}G^N$. On the other hand, we have 
$\apr^N_{(n)}\apr^N_{(n)} G^N$, and, hence, $\apr^N_{(n)}\neg\, G^N$.
Thus, $\apr_{(n)}^N\bot$. \emph{Quod non}. So, canceling our assumption, we find
  $\neg\,\apr^{\mathfrak J_nN}_{(n)}G^N$ and, thus,
$G^{\mathfrak J_nN}$.

Note the analogy of the present result with the result of Remark~\ref{smurfeltwee}.
\end{remark}

\subsection{Under guard}\label{mangosmurf}
In this subsection, we extract some information from the G\"odel sentences for $\apr_{(\mathfrak a^\ast)}$.
Let $\Im := \mathfrak I_{\mathfrak a^\ast} \cap \mathfrak C_{N_0, N_0K}$.

\begin{theorem}\label{wachtersmurf}
Suppose $A$ proves that $NK$ is a cut of $N_0K$. Suppose $A$ is consistent.
Then, $A \nvdash NK \subseteq  \mathfrak F_{N_0,N_0K}(\Im)$.
\end{theorem}

\begin{proof}
Suppose $A \vdash NK \subseteq  \mathfrak F_{N_0,N_0K}(\Im)$. Let $\widetilde N :=
\mathfrak F^{-1}_{N_0,N_0K}(NK)$. We note that $\mathfrak F_{N_0,N_0K}$ restricted to $\widetilde N$ is an
isomorphism between $\widetilde N$ and $NK$. Moreover, $\widetilde N \subseteq \mathfrak \Im$.
Let $G := G_{N,\mathfrak a^\ast}$.\footnote{Strictly speaking, we exceed our earlier framework here, since
$A$ need not prove that $N$ is a cut of $N_0$. However, it is easy to see that the marginal extensions
does no harm. The skeptical reader can always replace $N$ by $N\tupel{{\sf cut}_{N_0}(N)}N_0$, the interpretation
that is $N$ if $N$ is indeed an $N_0$-cut and that is $N_0$ otherwise.}

We reason in $A$. Suppose $\apr_{(a^\ast)}^{\widetilde N}G^N$.
It follows that $\apr_{(a^\ast)}^{NK}G^N$ and, hence $\neg\, G^{NK}$.
On the other hand, by guarded reflection, we have $G^{NK}$. A contradiction.
Hence, by canceling our assumption, $\neg\apr_{(a^\ast)}^{\widetilde N}G^N$, and so $G^{NK}$.

We return to the meta-language. We have shown that $A \vdash G^{NK}$. It follows that $A \vdash_{\mathfrak a^\ast} {\sf K}G^N$
and, hence, $A \vdash \apr^{NK}_{(\mathfrak a^\ast)} G^N$. Thus, $A \vdash \neg\, G^{NK}$. But this contradicts the consistency of
$A$.
\end{proof}

\noindent
Let $\Im^\ast :=  F_{N_0, N_0K}(\Im)$
and let  $\mathcal X :=  \verz{\Im^\ast \subset NK\mid \text{$NK$ is $A$-provably a cut of $N_0K$}}$.

\begin{theorem}\label{oppassmurf}
Suppose $A$ is consistent. Then, $A+ \mathcal X$ is also consistent.
\end{theorem}

\begin{proof}
Suppose $A+\mathcal X$ is inconsistent. Then, for some finite subset $\mathcal X_0$ of $\mathcal X$, 
we have $A \vdash \neg\,\bigwedge \mathcal X_0$. 

Let $\mathcal Y_0$ be the set of $N$
such that  $\Im^\ast \subset NK$ is in $\mathcal X_0$.
We have:
\begin{eqnarray*}
A \vdash \neg\,\bigwedge \mathcal X_0 & \iff & \bigvee_{N\in \mathcal Y_0} \Im^\ast \not\subset NK \\
& \iff & \bigvee_{N\in \mathcal Y_0}  NK\subseteq \Im^\ast\\
& \to & (\bigcap_{N\in\mathcal Y_0} N)K \subseteq \Im^\ast
\end{eqnarray*}
Writing $M$ for $\bigcap_{N\in\mathcal Y_0} N$, we find $A \vdash MK \subseteq \Im^\ast$. 
But this contradicts Theorem~\ref{wachtersmurf}.
\end{proof}

\noindent
Seeing that any $NK$ such that $NK:A \rhd {\sf S}^1_2$ has an initial segment that is $K$-internally definably isomorphic to
an $N_0K$-cut, Theorem~\ref{oppassmurf} tells us that there is an $A$-model $\mathcal M$ in which $\Im$ is below all $K$-internal
interpretations of ${\sf S}^1_2$. This implies, for example, that for any $B$ and any $m\geq {\sf max}(\rho(A),\rho(B))$, we have
$\mathcal M \models \opr^{\Im}_mB \to B^K$.

\subsection{Rosser?}
In this section we discussed the predicate $\apr$ and have shown that it has many good properties.
In  Remarks~\ref{smurfeleen}, \ref{smurfeltwee}, \ref{smurfeldrie} and in Subsection~\ref{mangosmurf}, we explored what information the
G\"odel sentences of $\apr$ could provide. However, as we seen our information until now makes $\apr$ too well behaved to
obtain a contradiction.

There is clearly an infinity of variations on the Rosser sentences and there is some hope that these might lead to the desired
contradiction.  

\section{In search of paradox}\label{para}
In this Section we study an attempt to prove Conjecture~\ref{hapjessmurf} that has some analogies to a
paradox, variants of which were independently found by Stephen Yablo and the author.

\subsection{Motivating remarks}\label{more}
We consider a finitely axiomatized, sequential $A$ with the Enayat property.
Let $N_0:{\sf S}^1_2\to A$ and let {\sf K} be defined with respect to $N_0$ as before.
Let us briefly dwell on the Liar for {\sf K}. By the G\"odel Fixed Point Lemma, we find
$L$ such that $A \vdash L \iff {\sf K}L$. We note that inside $A$ we have that
either $L$, $\neg\, L^K$, $L^{KK}$, \dots, or $\neg\, L$, $L^K$, $\neg\, L^{KK}$, \dots.
We also have $A\vdash L^{{\sf ID}\tupel{L}K}$ and  $A\vdash (\neg \, L)^{K\tupel{L}{\sf ID}}$,
and, thus,
$A \rhd (A+L)$ and $A \rhd (A+\neg\, L)$. So, $L$ is an Orey sentence for $A$.

Nothing dangerous seems to follow from the existence of $L$ since we only get the alternations.
Can be eliminate the alternations by stipulating that $L$ is false in all iterations of {\sf K}?
Consider $L^\ast$ such that $A \vdash L^\ast \iff \forall x\in \delta_{N_0}\, \neg\, {\sf K}^{x+1}L^\ast$.\footnote{We will worry 
about details of defining the iteration of {\sf K} later.}
We reason in $A$. Suppose $L^\ast$. It follows that $\neg\, {\sf K}L^\ast$ and, thus,
$\neg \, (L^\ast)^K$. We may conclude $(\exists x\in \delta_K\, {\sf K}^{x+1}L^\ast)^K$.
Now if $x$ were in the common cut $\mathfrak C_{N_0,N_0K}$, we would have our contradiction.
So, it follows that $N_0K$ contains an element above the common cut.
Suppose we start with $\neg\, L^\ast$, we get $\exists x\in \delta_K\, {\sf K}^{x+1}L^\ast$.
But what then? We do have $A \vdash (\exists x\in \delta_K\, {\sf K}^{x+1}L^\ast)^{K\tupel{L^\ast}{\sf ID}}$.
Can we tweak the argument in such a way that the existential claim gets a numerical witness?
If we could, we would have
  $A \vdash ({\sf K}^{\underline {n+1}}L^\ast)^{K\tupel{L^\ast}{\sf ID}}$, and, hence, 
 $A \vdash (L^\ast \wedge ({\sf K}^{\underline{n+1}}L^\ast)^K)  \vee  (L^\ast)^{K^{\underline{n+1}}}$.
 It follows that  
 $A \vdash (L^\ast \wedge {\sf K}^{\underline{n+2}}L^\ast)  \vee  (L^\ast)^{K^{\underline{n+1}}}$. 
  The first disjunct leads to a contradiction, so  $A \vdash  (L^\ast)^{K^{\underline{n+1}}}$.  
  By the fixed point equation, we find $A \vdash  (\neg\, {\sf K}L^\ast)^{K^{\underline{n+1}}}$, and, hence,
  $A \vdash  (\neg\, L^\ast)^{K^{\underline{n+2}}}$. On the other hand, since $K$ interprets $A$ in $A$,
  we find $A \vdash  ( L^\ast)^{K^{\underline{n+2}}}$. So, $A$ is inconsistent.
  
  The above program does not quite work. In the first place, we need to add some subtle details.
In the second place, and more importantly,  we need a substantial
  conjecture to make it work. However, this conjecture has some interest of its own.
 
 \begin{remark}
 The paradoxical reasoning sketched above is reminiscent of the reasoning in \cite{viss:sema04} concerning a descending
 hierarchy of truth predicates. (The article was first published in 1989 and the preprint appeared in 1985.)
 If you forget about the indices, this argument transforms into the well-known Yablo paradox, first published in \cite{yabl:trut85}. 
 \end{remark}

\subsection{The Small-is-very-small Principle and its variants}
The main idea of our strategy is the use of some kind of numerical existence principle that allows us to replace
a provable existence claim by a claim about an external number. This subsection provides the needed existence principle.

A theory is \emph{restricted} if all of its axioms have depth-of-quantifier-alternations complexity below a given standard number $k$.

The \emph{Small-is-very-small Principle}, or \emph{SIVS}, tells us that, if a restricted theory proves that a number with a certain property
exists in a sufficiently small cut (`is small') then the theory also believes that the number is standard (`is very small'). Here the relevant
small cut will depend on the property, or, more precisely the complexity of the formula defining the property.

\begin{theorem}[The Small-is-very-small Principle]\label{lolsmurf}
Consider a restricted sequential theory $U$ with bound $k$ and let $N_0:{\sf S}^1_2\lhd U$. Let 
$B$ be of the form $\exists x\in N_0\, B_0(x)$. Let $\ell$ be ${\sf max}(k,\rho(B))$ plus some constant $\mathfrak j$ for overhead.
\[ U \vdash \exists x\in \mathfrak I_\ell\, B_0(x) \;\;\; \To \;\;\;  \text{for some $m$ we have } U \vdash \exists x \leq \underline m\, B_0(x).\] 
\end{theorem}

\noindent
Here $\mathfrak I_\ell$ is the cut that was introduced in Subsection~\ref{custossmurf}. 
The numeral $\underline m$ is an $N_0$-numeral.
 
The proof of Theorem~\ref{lolsmurf} is given in full detail in \cite{viss:smal18}. Here we provide a quick sketch. 
Suppose $C$ is $\exists x\in \delta_{N_0}\, C_0$ and  $D$ is $\exists y\in \delta_{N_0}\, D_0$.
We write $C \leq D$ for $\exists x\in \delta_{N_0}\, (C_0 \wedge \forall y <^{N_0} x\, \neg\, D_0)$.

\begin{proof}[Proof-sketch] 
We work under the conditions specified in the theorem.

Suppose $U \vdash \exists x \in \mathfrak I_\ell \, B_0(x)$.
It follows that there is a finite subsystem $U_0$ of $U$ such that $U_0 \vdash \exists x \in \mathfrak I_\ell\, B_0(x)$.
We may assume that $U_0$ is sequential and verifies $({\sf S}^1_2)^{N_0}$.
 Now consider $R$ such that $U_0 \vdash R \iff B \leq \opr^{N_0}_{U_0,\ell}R$.

Reason in $U_0$. In case  $\opr^{N_0}_{U_0,\ell}R$ is not witnessed in $\mathfrak I_\ell$, we have $R$ 
by the fixed point equation. If
 $\opr^{N_0}_{U_0,\ell}R$ is witnessed in $\mathfrak I_\ell$, we have $R$ by reflection. We return to the meta-language.

We have shown $U_0 \vdash R$. By cut-elimination, we find $U_0 \vdash_\ell R$. Let $m$ witness $U_0 \vdash_\ell R$.
By $\Sigma_1$-completeness, we find, in $U_0$, that $B$ is witnessed below $m$.
\end{proof}

\begin{example}\label{kwakkelsmurf}
Let $U$ be a variant of {\sf PRA} in the arithmetical language.
Then $U$ is a consistent restricted sequential theory. Let $N_0$ be the identical interpretation. Since,
$U$ is also reflexive, we can find an interpretation $M$ such that 
$M:U\rhd U$ and $U \vdash \opr^{IM}_U \bot$, for all definable $N_0$-cuts $I$. It is obvious that we cannot have
\[U \vdash (\exists p \leq \underline m\;\, {\sf proof}_U(p,\bot))^M.\] So, we do not have an analogue for Theorem~\ref{lolsmurf}, if we embed
our existential sentence in a self-interpretation $M$. Intuitively, viewed from the standpoint of the world of $U$, the cuts definable
inside $M$, seen from the outside, cannot be as small as the cuts we have at the general level of the theory.
\end{example}

\noindent
We can escape the above example is we restrict ourselves to finitely axiomatized theories $A$ and restrict
$B$ to $(\Sigma_1^0)^{N_0}$-formulas. We then get the following internalized form of Theorem~\ref{lolsmurf}:

\begin{theorem}\label{pretsmurf}
Consider a consistent finitely axiomatized sequential theory $A$ and let $N_0:{\sf S}^1_2\lhd A$. Consider any number $n$.
Then, there is an $N_0$-cut $\mathcal I$, such that, for any $M:A \rhd A$ with $\rho(M)\leq n$, and for any $S\in \Sigma^0_1$, we have: if
$A \vdash S^{\mathcal IM}$, then $S$ is true. 
\end{theorem}

\noindent
We lost the restriction on the complexity of the existentially quantified formula from Theorem~\ref{lolsmurf} by a small trick. The proof of 
this result can be found in  \cite{viss:unpr93} or \cite{viss:faith05}. The essence of the trick in also given in the proof of Theorem~\ref{bankiersmurf}
below.

We note that Theorem~\ref{pretsmurf} contains both a restriction of the theory (it has to be finitely axiomatized) and on the
formula (it has to be $(\Sigma^0_1)^{N_0}$). One may wonder it the restriction to finitely axiomatized theories suffices. (We have seen,
in Example~\ref{kwakkelsmurf}, 
that the restriction to $(\Sigma^0_1)^{N_0}$-sentences does not suffice.)
Thus, we are lead to the following conjecture.
\begin{conjecture}\label{bonvivantsmurf}
Consider a finitely axiomatized sequential theory $A$ and let $N_0:{\sf S}^1_2\lhd A$. 
Consider any number $n$. There is an $N_0$-cut $\mathcal I_n$ such that,
for any sentence $B := \exists x\in N\, B_0(x)$ with $\rho(B)\leq n$ and any $M:A\rhd A$ with $\rho(M) \leq n$, we have:
\[ (\dag)\;\;\; A \vdash (\exists x\in \mathcal I_n\, B_0(x))^M \;\;\; \To \;\;\;  \text{for some $m$ we have } 
A \vdash (\exists x \leq \underline m\, B_0(x))^M.\] 
Here the $\underline m$ is an $N_0$-numeral.

We note that (\dag) is equivalent to:
\[ (\ddag)\;\;\; A \vdash (\exists x\in \mathcal I_n\, B_0(x))^M \;\;\; \To \;\;\;  \text{for some $m$ we have }
A \vdash \bigvee_{k\leq m}\, (B_0(\underline k))^M.\]
\end{conjecture}

\noindent
The conjecture could just turn out to be provable by a slightly more clever
Rosser argument than the ones I employed until now. My first attempts ran into the same kind
of problems as my attempts to prove the truth of Conjecture~\ref{hapjessmurf} directly: somewhere
a $K$ on an undesired place. How to get rid of it?

There is an interesting equivalent of Conjecture~\ref{bonvivantsmurf}.
\begin{conjecture}\label{olijkesmurf}
Consider a finitely axiomatized sequential theory $A$ and let $N_0:{\sf S}^1_2\lhd A$. 
Consider any number $n$. There is an $N_0$-cut $\mathcal J_n$ such that,
for any $\Sigma^0_1$-sentence $S$ and
for any sentence $C$ with $\rho(C)\leq n$ and any $M:A\rhd A$ with $\rho(M) \leq n$, we have:
$A \vdash (S^{\mathcal J_n} \vee C)^M \;\;\; \To \;\;\;  S \text{ is true, or }A\vdash C^M$. 
\end{conjecture}

\noindent
We note that, if we put $\bot$ for $C$, then we get something that is known, to wit Theorem~\ref{pretsmurf}.

\begin{theorem}\label{bankiersmurf}
Conjectures~\ref{bonvivantsmurf} and \ref{olijkesmurf} are equivalent. 
\end{theorem}

\begin{proof}
\emph{Conjecture~\ref{bonvivantsmurf} implies Conjecture~\ref{olijkesmurf}.}
Suppose we have Conjecture~~\ref{bonvivantsmurf}.
We remind the reader that ${\sf S}^1_2 \vdash x \in \mathfrak J_1 \to \exists y\, 2^x =y$.
It follows that ${\sf S}^1_2 \vdash S^{\mathfrak J_1} \to {\sf true}(\gn S)$, where
{\sf true} is the usual $\Sigma^0_1$-truth predicate. See \cite[Part C, Chapter V, 5(b)]{haje:meta91}
for a detailed treatment. We will write ${\sf true}(S)$ for ${\sf true}(\gn S)$.
We note that $\rho({\sf true}(S))$ is a fixed standard number $\mathfrak z$ independent of $S$.
Let ${\sf true}(x)$ be $\exists y \; {\sf true}_0(y,x)$, where ${\sf true}_0 \in \Delta_0$.

Let $\mathcal J_n := \mathfrak J_1\mathcal I_{{\sf max}(n,\mathfrak z+1)}$. where $\mathcal I_n$ is
provided by Conjecture~\ref{bonvivantsmurf}. Consider any $C$ and $K$ with complexities below
$n$. Let $n' := {\sf max}(n,\mathfrak z+1)$.
We have:
\begin{eqnarray*}
A \vdash (S^{\mathcal J_n} \vee C)^M & \To &   A \vdash(S^{\mathfrak J_1\mathcal I_{n'}} \vee C)^M \\
& \To & A \vdash({\sf true}^{\mathcal I_{n'}}( S) \vee C)^M \\
& \To & A \vdash (\exists y \in \mathcal I_{n'} \, ({\sf true}^{N_0}_0(y, S) \vee C))^M \\
& \To & \text{for some $m$, } A\vdash \bigvee_{k\leq m} ({\sf true}^{N_0}_0(\underline k, S) \vee C)^M \\
& \To & \text{$S$ is true or }A \vdash C^M.
\end{eqnarray*}
The last step uses that if $\neg\,{\sf true}_0(\underline k, S)$ is true, then $A \vdash \neg\,{\sf true}^{N_0M}_0(\underline k, S)$.

\emph{Conjecture~\ref{olijkesmurf} implies Conjecture~\ref{bonvivantsmurf}.} Suppose we have Conjecture~\ref{olijkesmurf}.
The proof uses an idea from \cite{frie:disj75}. Let $n$ be given. Let $n'$ be $n$ plus a constant for overhead.
We will be more specific about the choice of the constant later. Let $\mathcal I_n := \mathcal J_{n'}$.
Suppose we have $B$ and $M$, where $\rho(B)$ and $\rho(M)$ are $\leq n$.

Suppose $A \vdash (\exists x \in \mathcal I_n \, B_0(x))^M$. 
By the G\"odel Fixed Point Lemma, we find $R$ such that $A \vdash R \iff B \leq \opr^{N_0} R^M$.
We note that the complexity of $R$ is the complexity of $B$ plus a constant $\mathfrak y$ that just depends on $N_0$ and
the arithmetization of provability.  So, traveling back in time, we take $n' := n + \mathfrak y$. 
It is easy to see that $A \vdash  (\opr^{\mathcal I_n} R^M \vee R)^M$.
It follows that $\opr R^M$ is true or $A \vdash R^M$.
So, $A\vdash R^M$. Let $m$ be the G\"odel number of a witness of
$A \vdash R^M$. Then we find, that in $A \vdash {\sf proof}^{N_0M}(\underline m,R^M)$.
Combining this with $A \vdash R^M$, we find that, in $A$ and inside $M$, $B$ is witnessed below $m$.
\end{proof}

\noindent
In Appendix~\ref{binnensmurf}, we formulate weaker versions of Conjecture~\ref{bonvivantsmurf} and Conjecture~\ref{olijkesmurf}.

\subsection{Conjecture~\ref{bonvivantsmurf} implies Conjecture~\ref{hapjessmurf}}
We address the matter of defining ${\sf K}^y$. Let $A$ be finitely axiomatized, sequential and Enayat.
Suppose $N_0:A \rhd {\sf S}^1_2$ and $K:A \rhd (A+{\sf TB}^{-}_{N_0})$. Let $\mathfrak F:= \mathfrak F_{N_0,N_0K}$ be a definable 
isomorphism between the cut $\mathfrak C := \mathfrak C_{N_0,N_0K}$ of the $N_0$-numbers and its image in
the $N_0K$-numbers. 

We define the function $\gamma(y,x)$ as follows:
\begin{itemize}
\item
 $\gamma(0,x) = x$,
 \item
 $\gamma(y+1,x) := {\sf subst}(\gnum{{\sf K}(v_0)}, {\sf num}(\gamma(y,x)))$. 
 
 \noindent
 Here {\sf subst} is the substitution function and {\sf num} assigns to a number the
 G\"odel number of its numeral.
 \end{itemize}
 \emph{Par abus de langage}, we also write $\gamma$ for the arithmetization of $\gamma$.
 The function $\gamma$ is defined on the logarithmic numbers of $N_0$.
 Let $N_1$ be a logarithmic cut, e.g. $N_1 = \mathfrak J_1N_0$.
 We define, for $y\in \delta_{N_1}$,
 \begin{itemize}
 \item
 ${\sf K}^y D := (y=0 \wedge D) \vee \exists z<y\, (y=z+1 \wedge \exists u \in N_0\, (\gamma(z,\gnum D)= u \wedge {\sf K}(u)))$.
 \end{itemize}
Here is the main result of this section.

\begin{theorem}\label{grotesmurf}
The truth of Conjecture~\ref{bonvivantsmurf} implies the truth of Conjecture~\ref{hapjessmurf}.
\end{theorem}

\begin{proof}
We assume the truth of Conjecture~\ref{bonvivantsmurf}. Consider a finitely axiomatized, sequential theory $A$. 
Suppose $N_0:A \rhd {\sf S}^1_2$ and $K:A \rhd (A+{\sf TB}^{-}_{N_0})$.  We derive a contradiction.

 Let $n^\ast := {\sf max}(2 \rho(K),\rho({\sf K}(x)))+ 2$. We clearly may assume that $\mathcal I_{n^\ast}$ is a logarithmic
 cut in $N_0$, by shortening it when needed.  
 We use the fixed point lemma to obtain:
$A \vdash  L \iff {\sf K} (\forall w \in \mathcal I_{n^\ast} \neg\, {\sf K}^w L)$.
We note that $\rho(L)\leq n^\ast$, since, generally,  $\rho({\sf K}(\underline s)) =\rho({\sf K}(x)) +1$.

We have:
$A+ L \vdash  {\sf K}\neg\, L$ and hence  $A+L \vdash \neg\, L^K$.  
So, $A+L \vdash (\exists w \in \mathcal I_{n^\ast}\, {\sf K}^w L)^{KK}$.
Similarly, we have $A + \neg\, L \vdash (\exists w \in \mathcal I_{n^\ast}\, {\sf K}^w L)^{K}$.
Let $\widetilde K := KK\tupel{L} K$. It follows that $\widetilde K:A \rhd A$ and
$A \vdash (\exists w \in \mathcal I_{n^\ast}\, {\sf K}^w L)^{\widetilde K}$. 
We note that $\rho(\widetilde K) = {\sf max}(\rho(L), 2 \rho(K))+1$, so $\rho(\widetilde K) \leq n^\ast$.

We apply Conjecture~\ref{bonvivantsmurf} to obtain, for some $m$:
$A \vdash (\bigvee_{k\leq m} {\sf K}^{\underline k}L)^{\widetilde K}$.
Hence, \[ K+L \vdash \bigvee_{k\leq m} L^{K^{k+2}} \text{ and }K+\neg\, L \vdash \bigvee_{k\leq m} L^{K^{k+1}}.\]
It follows that (\ddag) $A \vdash \bigvee_{k \leq m+1} L^{K^{k+1}}$. 

Since $K:A \rhd A$, it follows that $A \vdash (\bigvee_{k \leq m+1} L^{K^{k+1}})^{K^{m+2}}$,
and, hence, that $A \vdash \bigvee_{k \leq m+1} L^{K^{m+k+3}}$.
On the other hand, by the definition of $L$, and  (\ddag), we find:
$A \vdash \bigwedge_{k \leq m+1} \neg \, L^{K^{m+k+3}}$.
So $A$ is inconsistent.
\end{proof}

\begin{remark}
We note that the construction $\widetilde K := KK\tupel{L}K$ preserves sentential restrictedness.
So, we need Conjecture~\ref{bonvivantsmurf} only for a very special kind of interpretation ---that does not even need to exist,
given the fact that we are looking for a \emph{reductio}.
\end{remark}

\appendix

\section{Vaught Set Theory continued}\label{nakaartsmurf}

We give the proof of Theorem~\ref{brutalesmurf}, i.e., we show that ${\sf VS} \rhd_{\sf dir} {\sf VS}^+$ via a
one-dimensional interpretation.

\begin{proof}
We define ${\sf PC}_0$ as the virtual class of all $x$ such that\footnote{`PC' stands for pre-cardinals.} 
\begin{enumerate}[1.]
\item
$x \sim x$
\item
$\forall y\,\forall z\,((x\sim y \wedge x\sim z) \to y \sim z)$
\item
$\forall y\,\forall z\,(x\sim y \sim z \to x \sim y)$
\end{enumerate}
Suppose $x$ is in ${\sf PC}_0$ and $x \sim y$. We show $y\in {\sf PC}_0$. 

\begin{enumerate}[ {A}d (1)]
\item
We have $x\sim y \wedge x\sim y$. Hence, by (2) for $x$, $y\sim y$.
\item
 Suppose $y \sim z$ and $y\sim u$. it follows that $x\sim y \sim z$ and $x \sim y\sim u$.
 Hence, by (3) for $x$, we find $x\sim z$ and $x \sim u$. Ergo, by (2), $z \sim u$.
 \item
 Suppose $y \sim z \sim u$. it follows that $x\sim y \sim z$, hence, by (3) for $x$, we find
 $x\sim z$. It follows that $x\sim z \sim u$,   hence, by (3) for $x$, we find 
 $x\sim u$. So we have $x\sim y$ and $x\sim u$, so, by (2) for $x$, we have $y \sim u$.
\end{enumerate}

The relation $\sim$ is an equivalence relation on ${\sf PC}_0$. The only thing to check is
symmetry. Consider $x\in {\sf PC}_0$ and suppose $x \sim y$. We have $x\sim y \wedge x \sim x$ and, so, by (2), $y \sim x$.

We define ${\sf PC}_1$ as the class of  $x$ in ${\sf PC}_0$ such that
 \[\forall u\,\forall v\,\forall f\, ((x\sim u\sim v \wedge f:u\sim v) \to u\sim f).\]
We note that $\sim$ is an equivalence relation on ${\sf PC}_1$.

We show that ${\sf PC}_1$ is closed under $\sim$.
Suppose ${\sf PC}_1(x)$ and $x\sim y$ and $y \sim u\sim v$ and $f:u\sim v$. Then, since ${\sf PC}_0(x)$, we have 
$x \sim u\sim v$ and $f:u\sim v$. So, we may conclude $u \sim f$. 

Finally, if ${\sf PC}_1(x)$ and $f: x\sim y$, then $x\sim x\sim y$ and $f:x\sim y$, and, so $x\sim f$, and, hence ${\sf PC}_1(f)$.

We now define $u\in^+ v$ iff $u\in v \wedge {\sf PC}_1(v)$. We claim that we have ${\sf VS}^+$ for $\in^+$.
We note that, whenever $x$ is in ${\sf PC}_1$, it defines the same set for $\in$ and for $\in^+$.
Also, any non-empty $\in^+$-set $x$ will be in ${\sf PC}_1$.

Clearly, any $\in$-finite set will be in ${\sf PC}_1$. Thus, it will have the same $\in^+$-elements. In other words,
finite sets are absolute with respect to our interpretation. It follows that we have the {\sf VS}-axioms for $\in^+$.

Since, Kuratowski pairs are constructed using finite sets, we easily see that Kuratowski pairs are absolute too.
It follows that, whenever, $f$ is an $\in$-function that is in ${\sf PC}_1$, then $f$ is also $\in^+$-function
with the same input-output behavior.

Let $u\sim^+v$ be  defined like $\sim$ with $\in$ replaced by $\in^+$.
We claim (\dag) for $x$ in ${\sf PC}_1$, we have $x\sim y$ iff $x\sim^+ y$.
First suppose $f:x\sim y$. Clearly, $f$ and $y$ will be in ${\sf PC}_1$, hence
$f:x\sim^+ y$. Conversely, suppose $g:x\sim^+y$. In case $x$ is $\in^+$-empty,
$x$ will be $\in$-empty, since $x\in {\sf PC}_1$. Hence, $x:x\sim y$. In case, $x$ is $\in^+$-non-empty,
also $y$ and $g$ will be $\in^+$-non-empty. Hence, they are in ${\sf PC}_1$.
It follows that $g:x\sim y$.

Consider any $x$. In case  $x$ is  $\in^+$-empty, we have $x:x\sim^+ x$.
In case $x$ is $\in^+$-non-empty, it is in ${\sf PC}_1$, hence $x\sim x$, and, so,
by (\dag), $x\sim^+ x$.

Consider any $x$ and y with $x\sim^+ y$. In case $x$ is $\in^+$-empty, we have $x:y\sim^+ x$.
Suppose $x$ is $\in^+$-non-empty. Then,  $x$ is in ${\sf PC}_1$. It follows, by (\dag), that $x\sim y$, and
hence, that $y$ is in ${\sf PC}_1$ and $y\sim x$. Hence, by (\dag), $y\sim^+ x$. 

Suppose $x\sim^+y\sim^+ z$. If $x$ is $\in^+$-empty, $y$ will be $\in^+$-empty and so will be $z$. Hence, $x:x\sim^+ z$.
Suppose $x$ is $\in^+$-non-empty. It follows that $x$ is in ${\sf PC}_1$. Hence, by (\dag), $x\sim y$, $y$ is in ${\sf PC}_1$ and
$y \sim z$. So $x\sim z$. Again by (\dag), $x\sim^+ z$.

Finally, suppose $f:x\sim^+ y$. In case $x$ is $\in^+$-empty, we find that $f$ is $\in^+$-empty, and, hence, $x: x\sim^+ f$.
Suppose $x$ is $\in^+$-non-empty, then so are $y$ and $f$. Then $x$, $y$ and $f$ are in ${\sf PC}_1$, and, hence,
$f:x \sim y$. It follows that $x \sim f$ and, hence, by (\dag), that $x\sim^+ f$.
\end{proof}

\noindent
Our next order of business is to prove Theorem~\ref{motorsmurf}, to wit that
${\sf VS}\rhd {\sf R}$ via a one-dimensional interpretation. There are two possible proofs. I will give (a sketch of) both.

\begin{proof}[First proof:]
By Theorem~\ref{brutalesmurf} it suffices to prove that ${\sf VS}^+ \rhd {\sf R}$.
The basic idea of our interpretation is to give the usual cardinal definitions of the
arithmetical operations whenever they work. When they do not work we set them to a
default value.

We write ${\sf pair}(x,y,z)$ for: $z$ represents a Kuratowski pair with first component $x$ and
second component $y$.

We call our translation $\rho$. We define:
\begin{itemize}
\item
$\delta_\rho(x) := (x=x)$,
\item
${x =_\rho y} := {x\sim y}$,
\item
${\sf Z}_\rho(x) := \forall y\, y \not\in x$,
\item
${\sf adj}(x,y,z) := \forall u\, (u\in z\iff (u\in x\vee u=y))$,
\item
${\sf S}_0(x,y)  := \exists z\, (z \not\in x \wedge {\sf adj}(x,z,y))$,
\item
${\sf S}_1(x,y) := \exists u\, \exists v\, ({\sf S}_0(u,v) \wedge  x\sim u \wedge v \sim y)$,
\item 
${\sf S}_2(x,y) := {\sf S}_1(x,y)  \wedge \forall z\, ( {\sf S}_1(x,z) \to y\sim z)$, 
\item
${\sf S}_\rho (x,y) := {\sf S}_2(x,y) \vee (\forall z\,  \neg\,  {\sf S}_2(x,z) \wedge x \sim y)$,
\item
${\sf U}(x,y,z) := \forall u\, (u \in z \iff (u\in x \vee u\in y))$,
\item
${\sf A}_0(x,y,z) := \forall u\, \neg\, (u\in x \wedge u\in y) \wedge {\sf U}(x,y,z)$,  
\item
${\sf A}_1(x,y,z) := \exists u\,\exists v\, \exists w\,  ({\sf A}_0(u,v,w) \wedge x\sim u \wedge y \sim v \wedge z \sim w)$,
\item
${\sf A}_2(x,y,z) := {\sf A}_1(x,y,z) \wedge \forall u\,  ( {\sf A}_1(x,y,u) \to z \sim u)$,
\item
${\sf A}_\rho(x,y,z) := {\sf A}_2(x,y,z) \vee (\forall u\, \neg\, {\sf A}_2(x,y,u) \wedge z \sim y)$,
\item
${\sf M}_0(x,y,z) := $\\
\hspace*{1cm}  $\forall u\in x\, \forall v \in y\,\exists w\in z\; {\sf pair}(u,v,w) \; \wedge$ \\
\hspace*{1.cm} $\forall w\in z\, \exists u\in x\, \exists v\in y\; {\sf pair}(u,v,w) \; \wedge$ \\
\hspace*{1cm} $ \forall p\in z\, \forall q\in z\, \forall u \in x\,\forall v\in y\, (({\sf pair}(u,v,p) \wedge {\sf pair}(u,v,q)) \to p=q)$,
\item
${\sf M}_1(x,y,z) := \exists u\,\exists v\, \exists w\,  ({\sf M}_0(u,v,w) \wedge x\sim u \wedge y \sim v \wedge z \sim w)$,
\item
${\sf M}_2(x,y,z) := {\sf M}_1(x,y,z) \wedge \forall u\,  ( {\sf M}_1(x,y,u) \to z \sim u)$,
\item
${\sf M}_\rho(x,y,z) := {\sf M}_2(x,y,z) \vee (\forall u\, \neg\, {\sf M}_2(x,y,u) \wedge z \sim y)$,
\end{itemize}
It is clear that on the standardly finite sets our operations behave as the ordinary successor, sum and product. Moreover,
$\leq$ defined as $x \leq y := \exists z\, (z+x=y)$ behaves as usual. Thus, $\rho$ carries an interpretation of {\sf R}. 
\end{proof}

\begin{remark}
We note that we could manipulate the interpretation of $\in$ further in order to interpret principles like:
\begin{itemize}
\item
$\forall x\, \forall y\, \forall z\, \forall u\, \forall v\, ((x\sim y \wedge u \not\in x \wedge v\not \in y \wedge 
{\sf adj}(x,u,z)) \to \exists w\, {\sf adj}(y,v,w))$.
\end{itemize}
Clearly, in the presence of such principles, we can build an interpretation following the above strategy that is
simpler.
\end{remark}

\begin{proof}[Second proof]
We interpret the theory of a category in {\sf VS}.
We define {\sf Ob} as the class of $x$ such that $\exists i:x \to x\; \forall y\in x\, i(y)=y$ and 
\[\forall y\, \forall z\, \forall f:x\to y\;\forall g: y\to z\;
\exists h: x\to z \;\forall u\in x\; f(g(u)) = h(u).\]
We define {\sf Morph} as the functions between the elements of {\sf Ob} and we take as identity on {\sf Morph} extensional sameness.
We define identity arrows and composition in the obvious way.
It is easy to see {\sf Ob} and {\sf Morph} with the chosen operations
define a category in {\sf VS} and that the standardly finite sets are in {\sf Ob}.

We now define ${\sf sum}(x,y,z)$ and ${\sf prod}(x,y,z)$ as the category-theoretical sum and product. We note that these are
partial operations but have verifiably good properties like commutativity and associativity.

Finally we define our interpretation, say $\nu$, of {\sf R} by taking:
\begin{itemize}
\item
$\delta_\nu := {\sf Ob}$,
\item
$=_\nu$ is isomorphism in our category,
\item
${\sf Z}_\nu(x) := \forall y\; y\not\in x$,
\item
${\sf A}_\nu(x,y,z) := {\sf sum}(x,y,z) \vee (\forall w\, \neg\, {\sf sum}(x,y,w) \wedge z=_\nu y)$,
\item
${\sf sing}(x) :=  \exists y\, \forall z \,(z\in x \iff z=y)$, 
\item
${\sf S}_\nu(x,y) := \exists z\, ({\sf sing}(z) \wedge {\sf A}_\nu(x,z,y))$,
\item
${\sf M}_\nu(x,y,z) := {\sf prod}(x,y,z) \vee (\forall w\, \neg\, {\sf prod}(x,y,w) \wedge z=_\nu y)$.
\end{itemize}
The rest of the proof is as expected.
\end{proof}

\section{Internal SIVS revisited}\label{binnensmurf}
We have a weaker version of the conjectured internal Small-is-very-small Principle
Conjecture~\ref{bonvivantsmurf}. 
This version is suggested by attempts to prove Conjecture~\ref{bonvivantsmurf}.
It looks like this.
Let $A$ be finitely axiomatized and sequential and let $N_0:A\rhd {\sf S}^1_2$.

\begin{conjecture}\label{ratelsmurf}
Let $n$ be given. Then, there is an $N_0$-cut $\mathcal I_n$, such that, for 
every sentence $B := \exists x\in N_0\, B_0(x)$ with $\rho(B) \leq n$, and, for every $M:A\rhd A$ with $\rho(M) \leq n$,
 we have: 
 \[  A \vdash (\exists x\in \mathcal I_n \, B_0(x))^M \;\; \To \;\; \text{there are $m$ and $k$ such that }
 A \vdash \bigvee_{i \leq m} \bigvee_{0<j \leq k} (B_0(\underline i))^{M^j}.\]
 
 \noindent Here $M^j$ means the $j$-fold iteration of $M$.
\end{conjecture}

\noindent
Our conjecture also has an equivalent variant (analogous to Conjecture~\ref{olijkesmurf}):

\begin{conjecture}\label{brilsmurf}
Consider any $n$. Then, there is an $N_0$-cut $\mathcal J_n$, such that for all
 $S\in \Sigma^0_1$ and for all $C$ with $\rho(C) \leq n$ and for all $M:A\rhd A$ with $\rho(M)\leq n$
 we have:  if $A\vdash (S^{\mathcal J_n} \vee C)^M$, then, for some $k$, we have $S$ is true or $A \vdash \bigvee_{0<j \leq k} C^{M^j}$.
\end{conjecture}

\begin{theorem}
Conjectures~\ref{ratelsmurf} and Conjecture~\ref{brilsmurf} are equivalent.
\end{theorem}

\noindent 
The proof is analogous to the proof of the equivalence of Conjectures~\ref{bonvivantsmurf}
and \ref{olijkesmurf}. The aim of the proof sketch below is mainly to highlight the differences.

\begin{proof}
\emph{Conjecture~\ref{ratelsmurf} implies Conjecture~\ref{brilsmurf}.}
Suppose we have Conjecture~~\ref{ratelsmurf}.

Let $\mathcal J_n := \mathfrak J_1\mathcal I_{{\sf max}(n,\mathfrak z+1)}$. where $\mathcal I_n$ is
provided by Conjecture~\ref{ratelsmurf}. Consider any $C$ and $K$ with complexities below
$n$. Let $n' := {\sf max}(n,\mathfrak z+1)$.
We have:
\begin{eqnarray*}
A \vdash (S^{\mathcal J_n} \vee C)^M & \To &   A \vdash(S^{\mathfrak J_1\mathcal I_{n'}} \vee C)^M \\
& \To & A \vdash({\sf true}^{\mathcal I_{n'}}( S) \vee C)^M \\
& \To & A \vdash (\exists y \in \mathcal I_{n'} \, ({\sf true}^{N_0}_0(y, S) \vee C))^M \\
& \To & \text{for some $m$ and $k$, } A\vdash \bigvee_{i\leq m}\bigvee_{0<j\leq k} ({\sf true}^{N_0}_0(\underline i, S) \vee C)^{M^j} \\
& \To & \text{$S$ is true or }A \vdash \bigvee_{0<j\leq k} C^{M^j}.
\end{eqnarray*}

\noindent
\emph{Conjecture~\ref{brilsmurf} implies Conjecture~\ref{ratelsmurf}.} Suppose we have Conjecture~\ref{brilsmurf}.
Let $n$ be given. Let $n'$ be $n+\mathfrak y$ and let $\mathcal I_n := \mathcal J_{n'}$.
Suppose we have $B$ and $M$, where $\rho(B)$ and $\rho(M)$ are $\leq n$.

Suppose $A \vdash (\exists x \in \mathcal I_n \, B_0(x))^M$. 
We write:  \[\opr\bigvee_j D_{0<j} := \exists p\, \exists y \leq |p| \, {\sf proof}(p,\bigvee_{0<i\leq y}D_i).\]
Here $|p|$ is the entier of the 2-logarithm of $p$. By the G\"odel Fixed Point Lemma, we find $R$ with
$A \vdash R \iff B \leq \opr^{N_0}\bigvee_{0<j} R^{M^j}$.

It is easy to see that $A \vdash  (\opr^{\mathcal I_n} \bigvee_{0<j} R^{M^j} \vee R)^M$.
It follows that $\opr \bigvee_{0<j} R^{M^j}$ is true or, for some $k$, we have $A \vdash \bigvee_{0<j\leq k}R^{M^j}$.
So, for some $k$, we have  $A\vdash \bigvee_{0<j\leq k} R^{M^j}$. Let $m$ be the G\"odel number of a witness of
$A \vdash \bigvee_{0<j\leq k}R^{M^j}$. Then,  $A \vdash {\sf proof}^{N_0M}(\underline m,\bigvee_{0<j\leq k}R^{M^j})$.
Combining this with $A \vdash \bigvee_{0<j\leq k}R^{M^j}$, we find the desired result.
\end{proof}

\noindent Here is the analogue of Theorem~\ref{grotesmurf}.

\begin{theorem}\label{zeussmurf}
The truth of Conjecture~\ref{ratelsmurf} implies the truth of Conjecture~\ref{hapjessmurf}.
\end{theorem}

\noindent The proof is just a trivial variant of the proof of Theorem~\ref{grotesmurf}.

\begin{proof}
We assume the truth of Conjecture~\ref{ratelsmurf}. Consider a finitely axiomatized, sequential theory $A$. 
Suppose $N_0:A \rhd {\sf S}^1_2$ and $K:A \rhd (A+{\sf TB}^{-}_{N_0})$.  We derive a contradiction.

 Let $n^\ast := {\sf max}(2 \rho(K),\rho({\sf K}(x)))+ 2$. We clearly may assume that $\mathcal I_{n^\ast}$ is a logarithmic
 cut in $N_0$, by shortening it when needed.  
 We use the fixed point lemma to obtain:
$A \vdash  L \iff {\sf K} (\forall w \in \mathcal I_{n^\ast} \neg\, {\sf K}^w L)$.
We note that $\rho(L)\leq n^\ast$, since, generally,  $\rho({\sf K}(\underline s)) =\rho({\sf K}(x)) +1$.

We have:
$A+ L \vdash  {\sf K}\neg\, L$ and hence  $A+L \vdash \neg\, L^K$.  
So, $A+L \vdash (\exists w \in \mathcal I_{n^\ast}\, {\sf K}^w L)^{KK}$.
Similarly, we have $A + \neg\, L \vdash (\exists w \in \mathcal I_{n^\ast}\, {\sf K}^w L)^{K}$.
Let $\widetilde K := KK\tupel{L} K$. It follows that $\widetilde K:A \rhd A$ and
$A \vdash (\exists w \in \mathcal I_{n^\ast}\, {\sf K}^w L)^{\widetilde K}$. 
We note that $\rho(\widetilde K) = {\sf max}(\rho(L), 2 \rho(K))+1$, so $\rho(\widetilde K) \leq n^\ast$.

We apply Conjecture~\ref{ratelsmurf} to obtain, 
\[ \text{for some $m$ and $k$: }
A \vdash (\bigvee_{i\leq  m}\bigvee_{0<j\leq k} {\sf K}^{\underline i}L)^{\widetilde K^j}.\]
Hence, \[ K+L \vdash \bigvee_{j\leq m+k} L^{K^{j+2}} \text{ and }K+\neg\, L \vdash \bigvee_{j\leq m+k} L^{K^{j+1}}.\]
It follows that (\ddag) $A \vdash \bigvee_{j \leq m+k+1} L^{K^{j+1}}$. 

Since $K:A \rhd A$, it follows that $A \vdash (\bigvee_{j \leq m+1} L^{K^{j+1}})^{K^{m+k+2}}$,
and, hence, that $A \vdash \bigvee_{j \leq m+1} L^{K^{m+k+j+3}}$.
On the other hand, by the definition of $L$, and  (\ddag), we find:
$A \vdash \bigwedge_{j \leq m+k+1} \neg \, L^{K^{m+k+j+3}}$.
So $A$ is inconsistent.
\end{proof}

\section{Conjectures and questions}\label{qanda}

\subsection{Conjectures}

{\small
\begin{enumerate}[c1.]
\item
No finitely axiomatized consistent Vaught  theory is Enayat. (Conjecture~\ref{smulsmurf}.)
Equivalently, we have the following conjecture.
Suppose $U$ is  a consistent Vaught theory. Then $\mathfrak T(U)$ is not quasi-finite.
(Conjecture~\ref{megasmurf}.)

 If this conjecture fails,
we conjecture that no finitely axiomatized consistent sequential  theory is Enayat. 
 (Conjecture~\ref{hapjessmurf}.)
\item
Suppose $A$ is finitely axiomatized and consistent and sequential. Let $N:{\sf S}^1_2\lhd A$.
Then, there is no extension of ${\sf S}^1_2$ that is mutually interpretable with $A+{\sf TB}^{-}_N$.
(Conjecture~\ref{extrasmurf}.)

It is a well known open question whether every sequential theory is mutually interpretable
with an extension-in-the-same-language of ${\sf S}^1_2$. Our conjecture provides a possible example
to illustrate a negative answer to this question.
\item
Let $U$ be Vaught.
and let $N:{\sf R}\lhd U$. Suppose $\alpha \jump_U{\sf TB}^{-}_N$. Then,
 $\alpha \jump_U{\sf USB}^-_{N}$. (Conjecture~\ref{partysmurf}.)
 \item
Suppose $A$ is a finitely axiomatized Vaught in signature $\Theta_0$. Let $N:A\rhd {\sf R}$. Suppose further that
 $\top \rhd_A {\sf TB}^-_N$. Then, there is a  $\beta$  such that
$\top\rhd_A \beta \jump_A {\sf TB}_N^-$.  

More generally, we may conjecture the following. Suppose $A$ is a finitely axiomatized Vaught theory and $\top \rhd_A V$.
Then, there is a $B$ such that $\top\rhd_A B \jump_A V$.
(Conjecture~\ref{yogasmurf}.)
\item
Consider a finitely axiomatized sequential theory $A$ and let $N_0:{\sf S}^1_2\lhd A$. 
Consider any number $n$. There is an $N_0$-cut $\mathcal I_n$ such that,
for any sentence $B := \exists x\in N\, B_0(x)$ with $\rho(B)\leq n$ and any $M:A\rhd A$ with $\rho(M) \leq n$, we have:
\[ (\dag)\;\;\; A \vdash (\exists x\in \mathcal I_n\, B_0(x))^M \;\;\; \To \;\;\;  \text{for some $m$ we have } 
A \vdash (\exists x \leq \underline m\, B_0(x))^M.\] 
Here the $\underline m$ is an $N_0$-numeral.

We note that (\dag) is equivalent to:
\[ (\ddag)\;\;\; A \vdash (\exists x\in \mathcal I_n\, B_0(x))^M \;\;\; \To \;\;\;  \text{for some $m$ we have }
A \vdash \bigvee_{k\leq m}\, (B_0(\underline k))^M.\]
(Conjecture~\ref{bonvivantsmurf}.)

There is an interesting equivalent of Conjecture~\ref{bonvivantsmurf}.
Consider a finitely axiomatized sequential theory $A$ and let $N_0:{\sf S}^1_2\lhd A$. 
Consider any number $n$. There is an $N_0$-cut $\mathcal J_n$ such that,
for any $\Sigma^0_1$-sentence $S$ and
for any sentence $C$ with $\rho(C)\leq n$ and any $M:A\rhd A$ with $\rho(M) \leq n$, we have:
$A \vdash (S^{\mathcal J_n} \vee C)^M \;\;\; \To \;\;\;  S \text{ is true, or }A\vdash C^M$. 
(Conjecture~\ref{olijkesmurf}.)
\item
Let $n$ be given. Then, there is an $N_0$-cut $\mathcal I_n$, such that, for 
every sentence $B := \exists x\in N_0\, B_0(x)$ with $\rho(B) \leq n$, and, for every $M:A\rhd A$ with $\rho(A) \leq n$,
 we have: {\footnotesize
 \[  A \vdash (\exists x\in \mathcal I_n \, B_0(x))^M \;\; \To \;\; \text{there are $m$ and $k$ such that }
 A \vdash \bigvee_{i \leq m} \bigvee_{0<j \leq k} (B_0(\underline i))^{M^j}.\]
 }
 
 \noindent Here $M^j$ means the $j$-fold iteration of $M$.
(Conjecture~\ref{ratelsmurf}.)

An equivalent conjecture runs as follows.
Consider any $n$. Then, there is an $N_0$-cut $\mathcal J_n$, such that for all
 $S\in \Sigma^0_1$ and for all $C$ with $\rho(C) \leq n$ and for all $M:A\rhd A$ with $\rho(M)\leq n$
 we have:  if $A\vdash (S^{\mathcal J_n} \vee C)^M$, then, for some $k$, we have $S$ is true or $A \vdash \bigvee_{0<j \leq k} C^{M^j}$.
(Conjecture~\ref{brilsmurf})
\end{enumerate}
}

\subsection{Questions}

{\small
\begin{enumerate}[q1.]
\item
 Suppose there is a finitely axiomatized, consistent Vaught theory that is Enayat. Can we show, under that assumption, that \emph{all}
finitely axiomatized, consistent, Vaught theories are Enayat theories?
(Question~\ref{voetbalsmurf}.)
\item
In Subsection~\ref{goochelsmurf}, we have shown that in the recursively enumerable sequential case, 
we can characterize Enayat theories in a coordinate-free
way. Not only is the question of Enayatness independent of the G\"odel numbering, but G\"odel numberings are not mentioned in the
characterization. Can we do something similar in the Vaught case?
(Question~\ref{hockeysmurf}.)
\item
Are there any interesting relations between theories, besides mutual interpretability, that preserve Enayatness?
(Question~\ref{tennissmurf}.)
\item
Is there an example of a finitely axiomatized theory $A$ with the $N$-Enayat property for some $N:{\sf Succ}_0\lhd A$, when we demand that the
G\"odel numbering is p-time computable? (Question~\ref{ruimtesmurf}.)
\item
Is there an example of a finitely axiomatized theory $A$ such that we have the Enayat property for all $N:{\sf Succ}_0\lhd A$?
(Question~\ref{rugbysmurf}.)
\item
Does ${\sf TB}_N^-$ have a restricted axiomatization over $U$?  
(Question~\ref{golfsmurf}.)
\item
 Can we show that, for no consistent sequential $U$, we have ${\sf TB}^-_N \jump_U {\sf USB}^-_N$?
(Question~\ref{damsmurf}.)
\item
Can we prove the non-existence of a finitely axiomatized consistent sequential uniform Enayat theory
 without a detour over the second incompleteness theorem?
(Question~\ref{schaaksmurf}.)
\item
Is there a finitely axiomatized  Vaught theory that is uniformly Enayat? Here uniformity is explicated
using ${\sf USB}^-$.
(Question~\ref{gosmurf}.)
\end{enumerate}
}
\end{document}